\documentclass{amsart}
\usepackage[utf8]{inputenc}
\usepackage[T1]{fontenc}
\usepackage{lmodern}
\usepackage{amsmath}
\usepackage{amssymb}
\usepackage{mathtools}
\usepackage{latexsym}
\usepackage[lite]{amsrefs}
\usepackage{nicefrac}
\usepackage{microtype}
\usepackage{color}
\usepackage{tikz-cd}
\usepackage{MnSymbol}
\usepackage{enumitem} 
\setlist[enumerate,1]{label=\textup{(\arabic*)}}

\usepackage[all]{xy}
\newdir{ >}{{}*!/-5pt/@{>}}

\usepackage[pdftitle={Non-Archimedean K-theory},
pdfauthor={Devarshi Mukherjee},
pdfsubject={Mathematics}
]{hyperref}

\BibSpec{book}{%
  +{}  {\PrintPrimary}                {transition}
  +{,} { \textit}                     {title}
  +{.} { }                            {part}
  +{:} { \textit}                     {subtitle}
  +{,} { \PrintEdition}               {edition}
  +{}  { \PrintEditorsB}              {editor}
  +{,} { \PrintTranslatorsC}          {translator}
  +{,} { \PrintContributions}         {contribution}
  +{,} { }                            {series}
  +{,} { \voltext}                    {volume}
  +{,} { }                            {publisher}
  +{,} { }                            {organization}
  +{,} { }                            {address}
  +{,} { \PrintDateB}                 {date}
  +{,} { }                            {status}
  +{}  { \parenthesize}               {language}
  +{}  { \PrintTranslation}           {translation}
  +{;} { \PrintReprint}               {reprint}
  +{.} { }                            {note}
  +{.} {}                             {transition}
  +{} { \PrintDOI}                   {doi}
  +{} { available at \url}            {eprint}
  +{}  {\SentenceSpace \PrintReviews} {review}
}

\renewcommand*{\PrintDOI}[1]{\href{http://dx.doi.org/\detokenize{#1}}{doi: \detokenize{#1}}}

\newcommand{\comment}[1]{}  

\theoremstyle{plain}
\newtheorem{theorem}{Theorem}[section]
\newtheorem{lemma}[theorem]{Lemma}
\newtheorem{corollary}[theorem]{Corollary}
\newtheorem{proposition}[theorem]{Proposition}
\theoremstyle{remark}
\newtheorem{remark}[theorem]{Remark}
\theoremstyle{definition}
\newtheorem{definition}[theorem]{Definition}
\newtheorem{example}[theorem]{Example}
\numberwithin{theorem}{section}

\newcommand\C{\mathbb C}
\newcommand\N{\mathbb N}
\newcommand\Q{\mathbb Q}

\newcommand\Z{\mathbb Z}

\newcommand{\coma}{\widehat}
\newcommand{\comb}{\overbracket[.7pt][1.4pt]}

\newcommand*{\tens}{\mathsf{T}}
\newcommand*{\jens}{\mathsf{J}}

\newcommand*{\tans}{\mathcal{T}}

\newcommand{\updagger}{\textup{\tiny\!\!\dagger}}

\newcommand{\Cont}{\mathrm{C}}
\newcommand{\cont}{\mathrm{cont}}
\newcommand{\defeq}{\mathrel{:=}} 
\newtheorem*{theorem*}{Theorem}
\newcommand*{\into}{\rightarrowtail}
\newcommand*{\onto}{\twoheadrightarrow}

\newcommand*{\ling}[1]{#1_\mathrm{lg}}

\DeclarePairedDelimiter{\abs}{\lvert}{\rvert}
\DeclarePairedDelimiter{\norm}{\lVert}{\rVert}
\DeclarePairedDelimiter{\gen}{\langle}{\rangle}
\DeclarePairedDelimiter{\floor}{\lfloor}{\rfloor}
\DeclarePairedDelimiterX{\setgiven}[2]{\{}{\}}{#1\,{:}\,\mathopen{}#2}

\newcommand\hot{\mathbin{\comb{\otimes}}}
\newcommand\haotimes{\mathbin{\coma{\otimes}}}

\DeclareMathOperator{\coker}{coker}

\DeclareMathOperator{\HOM}{HOM}
\DeclareMathOperator{\HA}{HA}

\newcommand*{\an}{\mathrm{an}}

\newcommand{\HAC}{\mathbb{HA}}

\newcommand{\op}{\mathrm{op}}
\newcommand{\ev}{\mathrm{ev}}
\newcommand{\odd}{\mathrm{odd}}
\newcommand{\nb}{\nobreakdash}
\newcommand{\dvr}{V}
\newcommand{\dvgen}{\pi}
\newcommand{\dvf}{F}
\newcommand{\resf}{\mathbb F}
\newcommand{\Alg}{\mathsf{Alg}}
\newcommand{\tf}{\mathrm{tf}}

\DeclareMathOperator{\Hom}{Hom}
\DeclareMathOperator{\HP}{HP}

\newcommand{\kk}{\mathsf{kk}^{\mathrm{an}}}

\begin{document}
\title{Nonarchimedean bivariant \(K\)-theory}

\author{Devarshi Mukherjee}
\email{dmukherjee@dm.uba.ar}

\address{Dep. Matemática-IMAS\\
 FCEyN-UBA, Ciudad Universitaria Pab 1\\
 1428 Buenos Aires\\ 
Argentina}

\begin{abstract}
We introduce bivariant \(K\)-theory for nonarchimedean bornological algebras over a complete discrete valuation ring \(\dvr\). This is the universal target for dagger homotopy invariant, matricially stable and excisive functors, similar to bivariant \(K\)-theory for locally convex topological \(\C\)-algebras and algebraic bivariant \(K\)-theory. As in the archimedean case, we use the universal property to construct a bivariant Chern character into analytic and periodic cyclic homology. When the first variable is the ground algebra \(\dvr\), we get a version of Weibel's homotopy algebraic \(K\)-theory, which we call \textit{stabilised overconvergent analytic \(K\)-theory}. The resulting analytic \(K\)-theory satisfies dagger homotopy invariance, stability by completed matrix algebras, and excision.  
\end{abstract}

\thanks{The author thanks Guillermo Corti\~nas, Guido Arnone, George Tamme, Ralf Meyer for helpful discussions, and the anonymous referee for useful comments. The author was funded by a Feodor Lynen Fellowship of the Alexander von Humboldt Foundation.}

\maketitle


\section{Introduction}

Throughout this article, let \(\dvr\) be a complete discrete valuation ring with uniformiser \(\dvgen\), fraction field \(\dvf\), and residue field \(\resf\). We assume throughout that~\(\dvf\) has
characteristic zero.

Bivariant \(K\)-theory was introduced by Kasparov (\cites{Kasparov:Invariants_elliptic,Kasparov:K-functors}) as a unification of complex topological \(K\)-theory and \(K\)-homology, with a view towards the Novikov conjecture. It has since been used in the classification of \(C^*\)-algebras, the Baum-Connes conjecture and in differential topology (\cites{Kasparov:Operator_K_applications,Kasparov:Novikov}). There are several equivalent ways of defining bivariant \(K\)-theory: in the form introduced by Kasparov, it is a category \(KK\) whose objects are separable \(C^*\)-algebras, and whose morphisms \(KK(A,B)\) are Hilbert \(B\)-modules with certain extra structure. The viewpoint that will be most relevant in this article is the approach due to Cuntz, which exhibits the morphism space as a noncommutative analogue of the stable homotopy category (\cites{cuntz1987new}). 

In analogy with the category of noncommutative motives (see \cite{tabuada2011guided}), bivariant \(K\)-theory is the universal target for functors on the category of separable \(C^*\)-algebras that are homotopy invariant, stable by compact operators and excisive for extensions with completely positive sections. Typical examples of such functors are asymptotic, local and analytic cyclic homology due to Michael Puschnigg (\cite{puschnigg2006asymptotic}) and Ralf Meyer (\cite{Meyer:HLHA}).  The source category of bivariant \(K\)-theory has since been enlarged to treat more general topological algebras, such as the Frechet algebra of smooth functions on a manifold, and the Weyl algebra with the fine topology (see \cites{Cuntz:Weyl,Cuntz-Thom:Algebraic_K}). This is done using \textit{classifying maps} of extensions of appropriate topological algebras with continuous linear sections. The most general class of algebras in which bivariant \(K\)-theory has been studied is the  category of complete bornological \(\C\)-algebras. Bivariant \(K\)-theory in this generality is discussed in (\cite{Cuntz-Meyer-Rosenberg}). Away from the topological setting, a purely algebraic version of bivariant \(K\)-theory is developed in (\cite{Cortinas-Thom:Bivariant_K}). Together with its equivariant (see \cites{MR3123759,MR4018774, cortinas2022, arnone2022graded}), graded and Hermitian versions, these algebraic bivariant \(K\)-theories have led to important results in the classification theory of Leavitt path algebras (\cites{cortinas2020homotopy,cortinas2022}).

The analytic \(K\)-theoretic invariants we propose use a combination of the tools developed in the operator algebraic bivariant \(K\)-theories and the purely algebraic version. This is justified by the fact that the topological algebras that arise in nonarchimedean geometry are completions or \textit{analytifications} of ordinary \(\dvr\)-algebras (see \cite{ben2022analytification}), which necessitates us to work in a general enough framework that allows for the passage between (homological) algebra and functional analysis. As in the archimedean case, the right source category to develop such theories is the category of complete, torsionfree bornological \(\dvr\)-algebras. Functional analysis in this context is developed in (\cites{Cortinas-Cuntz-Meyer-Tamme:Nonarchimedean, Meyer-Mukherjee:Bornological_tf}). 

This article follows a series of papers (\cites{Cortinas-Meyer-Mukherjee:NAHA, Meyer-Mukherjee:HA, Meyer-Mukherjee:HL}) that develop variants of periodic cyclic homology that have reasonable formal properties for nonarchimedean topological algebras. More concretely, the analytic cyclic homology complex is a functor
\[\HAC \colon \{\text{ Complete, torsionfree bornological } \dvr-\text{algebras}\} \to  \overleftarrow{\mathsf{Der}(\mathsf{Ind}(\mathsf{Ban}_\dvf))}\] into the homotopy category of pro-ind-systems of complexes of Banach \(\dvf\)-vector spaces. It satisfies homotopy invariance with respect to the algebra \(\dvr[t]^\updagger\) of overconvergent power series, stability with respect to suitably complete matrix algebras and excision for semi-split extensions of complete, torsionfree bornological \(\dvr\)-algebras. One of the main motivations of this article is to find the universal functor \[j \colon \{\text{ Complete, torsionfree bornological } \dvr-\text{algebras}\} \to \kk\] satisfying these properties. The existence and universality of such a functor means that \(\kk\)-equivalences automatically yield \(\mathbb{HP} (- \otimes \dvf)\) and \(\HAC\)-equivalences. This is important as although analytic cyclic homology only depends on its reduction mod \(\dvgen\) when restricted to a suitable subcategory, we cannot merely work in the bivariant algebraic \(kk\)-category relative to the residue field: \(kk\)-equivalences between \(\resf\)-algebras only yield \(\HAC\)-equivalences, and we do not yet know in what generality analytic and periodic cyclic homology agree. On the other hand, the more fundamental theory is of course periodic cyclic homology, and we use its computation for smooth, affinoid dagger algebras to construct Chern characters from homotopy algebraic \(K\)-theory to rigid cohomology. This is the analogue of the Chern character taking values in de Rham cohomology for manifolds. 

Finally, in future projects we also aim to study the Davis-L\"uck assembly map (\cite{davis1998spaces}) in the nonarchimedean analytic setting. In the purely algebraic case, recent work (\cite{ellis2022algebraic}) shows that the left hand side of the assembly map is a certain colimit of equivariant algebraic \(kk\)-groups. The assembly map is then a relationship between a variant of topological \(K\)-theory of \emph{completed} group algebras or crossed product algebras of discrete group actions, and equivariant bivariant analytic \(K\)-theory.

The article is organised as follows.

In Section \ref{sec:background}, we recall relevant background material on bornological functional analysis and topological \(K\)-theory in the nonarchimedean setting. Section \ref{sec:analytic-homotopy} introduces the \textit{overconvergent rigid \(n\)-simplex}, relative to which we define homotopies. This is defined as the simplicial ring \[[n] \mapsto \dvr\gen{\Delta^n}^\updagger \defeq \dvr[x_1,\dotsc,x_n]^\updagger/\gen{\sum x_ i - 1},\] where \(\dvr[x_1,\dotsc, x_n]^\updagger\) is the Monsky-Washnitzer algebra. We then describe the matrix algebras we seek stability results for. These include the \(\dvgen\)-adic completion \(\mathcal{M}^{\cont}\) of \(\mathbb{M}_\infty\), which is our main focus. 

In Section \ref{sec:definition-kk}, we define analytic \(kk\)-theory. The objects of this category are complete, torsionfree bornological \(\dvr\)-algebras, and its morphisms are \[\kk(A,B) = \varinjlim [\jens^n A, \mathcal{M}_{\infty}^\cont(B)^{\mathcal{S}^n}],\] where \(\jens\) denotes the noncommutative loops coming from the universal \textit{tensor algebra extension}. The bounded algebra homomorphisms in the inductive limit are induced by the classifying maps of the universal extension. As in the topological and algebraic setting, the definition is constructed in a manner that we have homotopy invariance, \(\mathcal{M}^\cont\)-stablility and excision for semi-split extensions of complete, torsionfree bornological algebras. Section \ref{sec:triangulated} shows that \(\kk\) is a triangulated category, where the distinguished triangles are isomorphic to diagrams of the form \[ \Omega(B) \to P_f \to A \to B,\] where \(P_f\) is the path algebra relative to a bounded algebra homomorphism \(f \colon A \to B\), and \(\Omega(B)\) is the \textit{loop functor} applied to \(B\).

In Section \ref{sec:analytic-K}, we study the relationship between our bivariant \(K\)-theory and various constructions defined previously. These include the \(KV\)-theories studied by Calvo and Hamida (\cites{hamida, calvo}), and the overconvergent version due to Tamme (\cite{tamme:thesis}). The definition of the \(KV\)-spectrum is arrived at by topologising the algebraic \(KV\)-spectrum. More concretely, for a Banach \(\dvr\)-algebra (resp. affinoid dagger algebra) \(A\), the \textit{topological (resp. analytic) \(KV\)-theory spectrum} is defined as  \[\mathsf{KV}^{\mathrm{top}}(A) \defeq \mathsf{BGL}^+(A \gen{\Delta^\bullet}) \text{ resp. } \mathsf{KV}^\an(A) \defeq \mathsf{BGL}^+(A\gen{\Delta^\bullet}^\updagger).\] The topological and analytic \(K\)-theory spectra coincide with the spectrum \(\mathsf{KV}(A/\dvgen A)\) associated to the reduction mod \(\dvgen\). We extend these definitions to nonconnective spectra using the Banach algebraic suspension \(\coma{\Sigma} \defeq \coma{\Gamma}/ \mathcal{M}^{\cont}\). The resulting theory, which we call \textit{overconvergent stabilised analytic \(K\)-theory} \(\tilde{K}^{\an, \updagger}(A) = K^{\an, \updagger}(\mathcal{M}^{\cont}(A))\) is the \(\mathcal{M}^{\cont}\)-stablisation of a version of Weibel's homotopy algebraic \(K\)-theory, which is the functor on the right hand side. The functor \(\tilde{K}^{\an, \updagger}\) defined on the category of complete, torsionfree bornological \(\dvr\)-algebras is dagger homotopy invariant, excisive and satisfies \(\mathcal{M}^{\cont}\)-stability by construction. The universal property of \(\kk\) yields a natural map \(\kk_n(\dvr, A) \to \tilde{K}_n^{\an, \updagger}(A)\), which we show is an isomorphism for each \(n \in \Z\) in Theorem \ref{thm:kk=KH}. Since the overconvergent analytic \(K\)-theory groups are an inductive limit of \(KV^{\an}\)-groups, they depend only on the reduction mod \(\dvgen\) of the algebra. In particular, since Weibel's homotopy algebraic \(K\)-theory satisfies \(\mathbb{M}_\infty\)-stability, we have \[\tilde{K}^{\an, \updagger}(A^\updagger) \cong K^{\an, \updagger}(A^\updagger) \cong KH(A/\dvgen A),\] whenever \(A^\updagger \subseteq \coma{A}\).

Finally, since bivariant analytic cyclic homology satisfies dagger homotopy invariance, excision and \(\mathcal{M}^{\cont}\)-stability, the universal property of \(\kk\) yields bivariant Chern characters 
\[
\kk_n(A,B) \to \HA_n(A,B),
\] for each \(n\). These specialise for \(A = \dvr\) to 
\[\tilde{K}^{\an,\updagger}_n(B) \overset{\mathrm{ch}_n}\to \HA_n(B),\] for each \(n \in \Z\). Since periodic cyclic homology also satisfies these properties, we also get bivariant Chern characters \(\kk_n(A,B) \to \HP_n(A \otimes \dvf,B \otimes \dvf)\), which specialise when \(A = \dvr\) to group homomorphisms \[\tilde{K}^{\an,\updagger}_n(B) \overset{\mathrm{ch}_n}\to \HP_n(B \otimes \dvf),\] for \(n \in \Z\). When \(B\) is the dagger completion of a smooth, finite-type \(\dvr\)-algebra, then we get Chern characters \(KH_n(B/\dvgen B) \to \HP_n(B \otimes \dvf) \cong \bigoplus_{j \in \Z} H_{\mathrm{rig}}^{n+2j}(B/\dvgen B, \dvf)\). This is analogous to the \textit{\(p\)-adic Chern character} from the \(p\)-completed, rationalised algebraic \(K\)-theory spectrum to the \(p\)-completed, rationalised periodic cyclic homology spectrum
\[ K(A/\dvgen A, \Q_p) \to \HP(A, \Q_p)\]  constructed in \cite{antieau2020beilinson}*{Definition 2.14}. 


\section{Background}\label{sec:background}

\subsection{Preliminaries from bornological analysis}

As in~\cites{Cortinas-Cuntz-Meyer-Tamme:Nonarchimedean,
  Meyer-Mukherjee:Bornological_tf,
  Cortinas-Meyer-Mukherjee:NAHA,Meyer-Mukherjee:HA}, we use the
framework of bornologies to do nonarchimedean analysis.
A \emph{bornology} on a set~\(X\) is a collection of its subsets,
which are called \emph{bounded subsets}, such
that finite subsets are bounded and finite unions and subsets of
bounded subsets remain bounded.

A \emph{bornological \(\dvr\)\nb-module} is a
\(\dvr\)\nb-module~\(M\) with a bornology such that every bounded
subset is contained in a bounded \(\dvr\)\nb-submodule.  We call a
\(\dvr\)\nb-module map \(f \colon M \to N\) \emph{bounded} if it
maps bounded subsets of~\(M\) to bounded subsets of~\(N\).  A
bornological \(\dvr\)\nb-algebra is a bornological
\(\dvr\)\nb-module with a bounded multiplication map.  A
\emph{complete} bornological \(\dvr\)\nb-module is a bornological
\(\dvr\)\nb-module in which every bounded subset is contained in a
bounded, \(\dvgen\)\nb-adically complete \(\dvr\)\nb-submodule.
Every bornological \(\dvr\)\nb-module~$M$ has a
completion~$\comb{M}$ (see
\cite{Cortinas-Cuntz-Meyer-Tamme:Nonarchimedean}*{Proposition~2.14}).

\begin{example}
  The most basic example of a bornology on a \(\dvr\)\nb-module is
  the \emph{fine bornology}, which consists of those subsets that
  are contained in a finitely generated \(\dvr\)\nb-submodule.  Any
  fine bornological \(\dvr\)\nb-module is complete.  By default, we
  equip modules over the residue field~\(\resf\) with the fine
  bornology.
\end{example}

\begin{definition}[\cite{Meyer-Mukherjee:Bornological_tf}*{Definition~4.1}]
  \label{def:bornologically_tf}
  We call a bornological \(\dvr\)\nb-module~\(M\) (bornologically)
  \emph{torsionfree} if multiplication by~\(\dvgen\) is a
  bornological embedding, that is, $M$ is algebraically torsionfree
  and
  \(\dvgen^{-1} \cdot S \defeq \setgiven{x \in M}{\dvgen x \in S}\)
  is bounded for every bounded subset \(S \subseteq M\).  A
  \(\dvr\)\nb-module with the fine bornology is bornologically
  torsionfree if and only if it is torsionfree in the purely
  algebraic sense.  For the rest of this article, we briefly write
  ``torsionfree'' instead of ``bornologically torsionfree''.
\end{definition}

\begin{lemma}\label{lem:complete-category}
The category of complete, bornologically torsionfree \(\dvr\)-modules is complete. 
\end{lemma}

\begin{proof}
It only needs to be checked that this category is closed under kernels and products. Given a map \(f \colon M \to N\) of complete bornological \(\dvr\)-modules, its kernel is a closed and hence a complete bornological \(\dvr\)-submodule of \(M\) by \cite{Meyer-Mukherjee:Bornological_tf}*{Theorem 2.3}. For products, consider a family \((M_i)_{i \in I}\) of complete, bornological \(\dvr\)-modules. Then the product bornology on \(\prod_{i \in I} M_i\) turns it into a complete bornological \(\dvr\)-module. The kernel of a map between bornologically torsionfree \(\dvr\)-modules is a submodule with the subspace bornology, and is hence bornologically torsionfree by \cite{Meyer-Mukherjee:Bornological_tf}*{Lemma 4.2}. Finally, if \((M_i)_{i \in I}\) is a family of bornologically torsionfree \(\dvr\)-modules, then there are bornological embeddings \(M_i \subseteq M_i \otimes \dvf\) for each \(i\) by \cite{Meyer-Mukherjee:Bornological_tf}*{Proposition 4.3}. These are kernels in the category of bornological \(\dvr\)-modules, which therefore commute with products. So there is a bornological embedding \(\prod_{i \in I} M_i \subseteq \prod_{i \in I} M_i \otimes \dvf\), from which we conclude that \(\prod_{i \in I} M_i\) is bornologically torsionfree.
\end{proof}

We will use the following on several occasions in the paper:

\begin{lemma}\cite{Meyer-Mukherjee:HA}*{Lemma 2.7}
  \label{lem:tensor-exact}
  Let \(f \colon M \to N\) be a bornological embedding between
  complete, bornologically torsionfree \(\dvr\)-modules.  Then for
  an arbitrary complete, torsionfree bornological
  \(\dvr\)\nb-module~\(D\), the induced map
  \(f \otimes 1_D \colon M \hot D \to N \hot D\) is a bornological
  embedding.
\end{lemma}

\begin{proof}
  Since~\(D\) is complete and bornologically torsionfree, we may
  write it as an inductive limit
  \(D \cong \varinjlim_i \Cont_0(X_i,\dvr)\) of unit balls of Banach
  \(\dvf\)\nb-vector spaces, with bounded, injective structure maps.
  This is \cite{Cortinas-Meyer-Mukherjee:NAHA}*{Corollary~2.4.3}.
  For each~\(i\), the embedding~\(f\) induces a bornological
  embedding from
  \(\Cont_0(X_i,\dvr) \haotimes M \cong \Cont_0(X_i,M)\) to
  \(\Cont_0(X_i, \dvr) \haotimes N \cong \Cont_0(X_i,N)\) by Lemma
  \ref{lem:tensor-exact}.  The
  embedding property is preserved by taking inductive limits.  Since
  the structure maps of the inductive systems \((\Cont_0(X_i,M))_i\)
  and \((\Cont_0(X_i,N))_i\) are injective, the inductive limit is
  already separated, so that no separated quotient occurs.  Since
  the completed projective tensor product commutes with filtered
  colimits, we obtain the desired bornological embedding from
  \(D \hot M \cong \varinjlim_i \Cont_0(X_i, M)\) into
  \(\varinjlim_i \Cont_0(X_i,N) \cong D \hot N\).
\end{proof}

\begin{definition}[\cites{Cortinas-Cuntz-Meyer-Tamme:Nonarchimedean,
    Meyer-Mukherjee:Bornological_tf}]
  \label{def:dagger_algebra}
  We call a bornological \(\dvr\)\nb-algebra~\(D\)
  \emph{semidagger} if, for every bounded subset \(S \subseteq D\),
  the \(\dvr\)\nb-submodule \(\sum_{i=0}^\infty \dvgen^i S^{i+1}\)
  is bounded in~\(D\).  A complete, torsionfree, semidagger
  bornological \(\dvr\)\nb-algebra is called a \emph{dagger
    algebra}.
\end{definition}

\begin{example}
  \label{exa:resf_semidagger}
  Any \(\resf\)\nb-algebra \(A\) with the fine bornology is semidagger and
  complete. That it is semi-dagger follows from the fact that since \(A\) is viewed as a \(\dvr\)-module via the quotient map \(\dvr \to \resf\), for any finitely generated submodule \(S \subseteq A\), we have \(S^{\diamond} = \sum_{n=0}^\infty \dvgen^n S^{n+1} = S\). 
\end{example}

\begin{example}
  \label{exa:Banach_algebra}
  Let~\(B\) be a Banach \(\dvf\)\nb-algebra.  We assume the norm
  of~\(B\) to be submultiplicative.  Let \(D\subseteq B\) be the
  unit ball.  Then \(D\cdot D\subseteq D\), and~\(D\) becomes a
  \(\dvgen\)\nb-adically complete, torsionfree \(\dvr\)\nb-algebra.
  Conversely, if such an algebra~\(D\) is given, then
  \(D\hookrightarrow D\otimes \dvf\) and there is a unique norm on
  \(D\otimes \dvf\) with unit ball~\(D\).

  Let~\(D\) be the unit
  ball of a Banach \(\dvf\)\nb-algebra as above.  Then we call~\(D\)
  with the bornology where all subsets are bounded a \emph{Banach
    \(\dvr\)\nb-algebra}.  This bornology makes~\(D\) a dagger
  algebra.
\end{example}

\begin{definition}(\cite{Cortinas-Cuntz-Meyer-Tamme:Nonarchimedean})
  Any bornology on a \(\dvr\)\nb-algebra~\(D\) is contained in a
  smallest semidagger bornology, namely, the bornology generated by the
  \(\dvr\)\nb-submodules of the form
  \(\sum_{i=0}^\infty \dvgen^i S^{i+1}\), where \(S \subseteq D\) is
  bounded in the original bornology.  This is called the
  \emph{linear growth bornology}.  We denote~\(D\) with the linear
  growth bornology by~\(\ling{D}\).
\end{definition}

If~\(D\) is torsionfree, then the completion
\(D^\dagger \defeq \comb{\ling{D}}\) is a dagger algebra (see
\cite{Meyer-Mukherjee:Bornological_tf}*{Proposition~3.8} or, in
slightly different notation,
\cite{Cortinas-Cuntz-Meyer-Tamme:Nonarchimedean}*{Lemma 3.1.12}).

\begin{definition}\label{def:fine-mod-p}(\cites{Meyer-Mukherjee:HA, Meyer-Mukherjee:HL})
  A bornological \(\dvr\)\nb-module~\(M\) is called \emph{fine
    mod~\(\dvgen\)} if the quotient bornology on \(M/\dvgen M\) is
  the fine one.  Equivalently, any bounded subset is contained in
  \(\mathcal{F}+\dvgen M\) for a finitely generated \(\dvr\)\nb-submodule
  \(\mathcal{F}\subseteq M\).
\end{definition}

\begin{example}
Any nuclear bornological \(\dvr\)-algebra (\cite{Meyer-Mukherjee:HL}*{Definition 3.1}) is fine mod \(\dvgen\). Examples of such algebras include torsionfree \(\dvr\)-algebras with the fine bornology, and any torsionfree \(\dvr\)-module with the bornology where a subset \(S\) is bounded if it is contained in a bounded \(\dvr\)-module \(T\) and there is a null sequence \((t_n) \in T\) such that  \(S = \setgiven{s = \sum_{n=0}^\infty c_n t_n}{(c_n) \in l^\infty(\N,\dvr), s \text{ converges in } T}\). 
\end{example}

\subsection{Topological \(K\)-theories in the nonarchimedean context}\label{subsec:existing-theories}

In this subsection, we recall some already existing constructions of topological \(K\)-theory in the context of nonarchimedean Banach algebras, due to Calvo and Hamida (\cites{hamida, calvo}). These are defined by modifying the interval objects of homotopy invariant versions of algebraic \(K\)-theory, namely \(KV\) (\cite{kv}) and \(KH\)-theory (\cite{kh}), taking into account the topology on the algebra. For nonarchimedean Banach algebras, a natural choice of interval object is the algebra of power series 
\[
\dvr\gen{x_1,\dotsc,x_n} = \setgiven{\sum_{I \subseteq \N^n} c_I x^I}{\lim_I \abs{c_I} = 0}
\] convergent on the unit polydisc. Equipped with the Gauss norm \(\abs{\sum c_I x^I} \defeq \max_I \abs{c_I}\), this is a Banach \(\dvr\)-algebra. One then  defines a simplicial ring 
\[ \dvr\gen{\Delta^\bullet} \defeq [n] \mapsto \dvr\gen{x_0,\dotsc, x_n}/ \gen{\sum_{i=0}^n x_i - 1},\] where the \(0\)-th term is just \(\dvr\). Now for any Banach \(\dvr\)-algebra, we can form the simplicial ring \(A\gen{\Delta^\bullet} \defeq A \hot \dvr\gen{\Delta^\bullet}\), where \(\hot\) denotes the completed, projective tensor product in the category of Banach \(\dvr\)-modules. Using this, they define the \textit{topological \(K\)-theory} of a unital Banach algebra \(A\) as the spectrum
\[\mathsf{K}^{\mathrm{top}}(A) \defeq \mathsf{K}(A \gen{\Delta^\bullet}),\] and its homotopy groups \(K_n^{\mathrm{top}}(A)\) for \(n\geq 1\) as the topological \(K\)-theory groups of \(A\). The extension to non-unital algebras, as before, involves taking the homotopy fibre \(\mathsf{K}^{\mathrm{top}}(A) \defeq \mathsf{fib}(\mathsf{K}^{\mathrm{top}}(\tilde{A}) \to \mathsf{K}^{\mathrm{top}}(\dvr))\) of the unitalisation \(\tilde{A} = A \oplus \dvr\) with the product topology.

Hamida's topological \(K\)-theory is homotopy invariant with respect to the closed unit disc \(\mathbb{A}^{1,\an}(1) = \mathrm{Sp}(\dvr\gen{x})\) with a \emph{fixed} radius (say radius \(1\)). The recent work of Kerz-Saito-Tamme \cite{MR4012551} develops a theory for Banach algebras over the fraction field \(\dvf\) that is homotopy invariant with respect to discs of all radii simultaenously. That is, their interval object is \(\mathbb{A}^{1,\an} = \mathrm{colim}_r \mathbb{A}^{1,\an}(r)\), where \(\mathrm{colim}_r \mathbb{A}^{1,\an}(r) = \mathrm{Sp}(\dvf\gen{x}_r)\) for \(\dvf \gen{x}_r = \setgiven{\sum_{n \in \N}c_n x^n}{\lim \abs{c_n}r^n = 0}\). For each such radius \(r>0\), they define the simplicial algebra \(A\gen{\Delta^\bullet}_r = A \hot \dvf\gen{\Delta^\bullet}_r\), which form a projective system \[r \mapsto A\gen{\Delta^\bullet}_r,\] of simplicial algebras upon varying the radius. Taking their connective algebraic \(K\)-theory yields a pro-spectrum \[k^{\mathrm{an}}(A) \defeq \lim_r \mathsf{K}(A \gen{\Delta^\bullet}_r),\] which they call the \textit{connective analytic \(K\)-theory} of an affinoid algebra \(A\). The extension to nonconnective pro-spectra involves a delooping construction, which the interested reader can find in \cite{MR4012551}*{Section 4.4}. Since our main motivation is to develop bivariant \(K\)-theory for torsionfree \(\dvr\)-algebras, we do not attempt to specialise our theory to that in (\cite{MR4012551}), but rather only construct an integral version of it.

\section{Analytic homotopies, stability and extensions}\label{sec:analytic-homotopy}

In what follows, let \(\Alg_V^\tf\) denote the category of complete, bornologically torsionfree \(\dvr\)-algebras. Its objects are complete, bornologically torsionfree \(\dvr\)-algebras and its morphisms are bounded \(\dvr\)-algebra homomorphisms. 

\subsection{Analytic homotopies}

Consider the overconvergent analytic \(n\)-simplex defined by 
\[\dvr\gen{\Delta^\bullet}^\updagger \defeq [n] \mapsto \dvr\gen{\Delta^n}^\updagger,\] where \(\dvr \gen{\Delta^n}^\updagger = \dvr[t_0,\dotsc,t_n]^\updagger/ \gen{\sum_{i = 0}^n t_i - 1}\). This is a simplicial object in the category \(\Alg_\dvr^\tf\).

\begin{lemma}\label{lem:3}
The simplicial ring \(\dvr\gen{\Delta^\bullet}^\updagger\) is weakly contractible. 
\end{lemma}
\begin{proof}
Denote by \((d_i)_{i \geq 0}\) the face maps of the simplicial group \(\dvr\gen{\Delta^\bullet}^\updagger\). These are defined as \[d_i(f)(t_0,\dotsc,t_n) = f(t_0,\dotsc,t_{i-1},0,t_i,\dotsc,t_n)\] for \(f \in \dvr\gen{\Delta^n}^\updagger\). By the Yoneda lemma, a \(1\)-simplex \(x_0 \in \dvr\gen{\Delta^1}^\updagger\) corresponds to a morphism of simplicial sets \(f_{x_0} \colon \Delta[1] \to \dvr\gen{\Delta^\bullet}^\updagger\) such that \(f_{x_0}(\delta_0) = d_1(x_0) = 1\) and \(f_{x_0}(\delta_1) = d_0(x_0) = 0\) for \(\delta_0\) and \(\delta_1 \in \Hom_{\Delta}([0],[1])\). The required simplicial homotopy is given by \(\dvr\gen{\Delta^\bullet}^\updagger \times \Delta[1] \to \dvr\gen{\Delta^\bullet}^\updagger\), \((g,t) \mapsto f_{x_0}(t)\cdot g\). This shows that the identity on \(\dvr\gen{\Delta^\bullet}^\updagger\) is null-homotopic, as required.
\end{proof}

Using the overconvergent analytic simplex, we simplicially enrich our category. Let \(A\) be a complete, bornologically torsionfree \(\dvr\)-algebra. We define the simplicial ring \(A\gen{\Delta^\bullet}^\updagger \colon [n] \mapsto A \hot \dvr \gen{\Delta^n}^\updagger\). The \textit{mapping space} bifunctor \(\Hom_{\Alg_V^\tf} \colon {\Alg_V^\tf}^\op \times \Alg_V^\tf \to \mathbb{S}\) is defined in the obvious way:
\begin{equation}\label{eq:mapping-space}
\Hom_{\Alg_V^\tf}(A,B) \colon [n] \mapsto \Hom_{\Alg_V^\tf}(A, B\gen{\Delta^n}^\updagger),
\end{equation} wherein the composition rule for three algebras in \(\Alg_V^\tf\) is easy to define. With the following lemma, we conclude that \(\Alg_\dvr^\tf\) is a simplicial category.

\begin{lemma}\label{lem:simplicial-category}
For \(A \in \Alg_V^\tf\), the contravariant representable functor \(\Hom_{\Alg_V^\tf}(-,A)\) has a left adjoint defined by the functor \[A^{(-)} \colon \mathbb{S} \to {\Alg_V^\tf}^\op, \quad A^X \defeq \lim_{\Delta^n \to X} A\gen{\Delta^n}^\updagger,\] for a simplicial set \(X\).
\end{lemma}

\begin{proof}
This is part of a general construction. Let \(\mathcal{C}\) be a locally small category with colimits, and \(F \colon \Delta \to \mathcal{C}\) a covariant functor. Then the functor \(R \colon \mathcal{C} \to \mathbb{S}\) defined by \(R(c)[n] \defeq \Hom_{\mathcal{C}}(F([n]), c)\) has a left adjoint, given by the Kan extension of \(F\) along the Yoneda embedding \(y \colon \Delta \to \mathbb{S}\). The resulting object is a coend, which is a colimit and hence exists by hypothesis. In our case, the category \(\mathcal{C}\) is \({\Alg_\dvr^\tf}^\op\), and the functor \(F\) is the contravariant functor \(\Delta^\op \to \Alg_\dvr^\tf\), \([n] \mapsto A\gen{\Delta^\bullet}^\updagger\). A coend in \(\mathcal{C}\) is an end in \(\mathcal{C}^\op\), which exists in our case since \(\Alg_\dvr^\tf\) has all limits by Lemma ~\ref{lem:complete-category}. \qedhere
\end{proof}

\begin{lemma}\label{lem:2}
For \(X \in \mathbb{S}\) and \(B \in \Alg_\dvr^\tf\), we have \(B^X \cong \Hom_{\mathbb{S}}(X, B\gen{\Delta^\bullet}^\updagger)\).  
\end{lemma}

\begin{proof}
By the adjuction in Lemma \ref{lem:simplicial-category}, we have \[\Hom_{\Alg_\dvr^\tf}(A,B^X) \cong \Hom_{\mathbb{S}}(X, \Hom_{\Alg_\dvr^\tf}(A,B)),\] for \(X \in \mathbb{S}\) and \(A\), \(B \in \Alg_\dvr^\tf\). Then for \(A = t \dvr[t]\) with the fine bornology, bounded algebra homomorphisms \(A \to C\) to a complete, bornologically torsionfree algebra \(C\) are in bijection with bounded \(\dvr\)-linear maps \(\dvr \to C\), which in turn are in bijection with \(C\). This applies to \(B^X\) on the left hand side, and to each of the \(n\)-simplicies \(\Hom_{\Alg_\dvr^\tf}(A, B\gen{\Delta^n}^\updagger)\) on the right hand side to yield the desired result.
\end{proof}

Now let \(\mathbb{S}_*\) denote the category of pointed simplicial sets. Suppose \((K,*) \in \mathbb{S}_*\), we define \begin{multline*}
A^{(K,*)} \defeq \mathsf{Hom}_{\mathbb{S}_*}((K,*), A\gen{\Delta^\bullet}^\updagger) \\
\cong \ker(\Hom_{\mathbb{S}}(K,A) \to \Hom_{\mathbb{S}}(*,A)) \cong \ker(A^K \to A)
\end{multline*} for an algebra \(A \in \Alg_\dvr^\tf\).

We will need the following at several points in the paper:

\begin{lemma}\label{lem:4}
Let \(K\) be a finite simplicial set, \(*\) a base point of \(K\), and \(A\) a complete, bornologically torsionfree \(\dvr\)-algebra. Then there are natural isomorphisms \[\dvr^K \hot A \cong A^K, \quad \dvr^{(K,*)} \hot A \cong A^{(K,*)}.\]
\end{lemma}

\begin{proof}
We first consider the unpointed part. Here we need to show that the canonical map \[V^K \hot A = (\underset{\Delta^n \to K}\lim \dvr \gen{\Delta^n}^\updagger ) \hot A \to \underset{\Delta^n \to K}\lim (\dvr\gen{\Delta^n}^\updagger \hot A) = A^K\] is an isomorphism. That is, tensoring with \(A\) preserves limits. Since \(K\) is a finite simplicial set, it suffices to show that \(A \hot -\) commutes with finite limits, or, finite products and kernels. Since the category of complete bornological modules is additive, finite products are finite direct sums, which the completed bornological tensor product preserves. That tensoring with \(A\) preserves kernels follows from Lemma \ref{lem:tensor-exact}. The pointed part follows from the fact that the extension of complete bornological \(\dvr\)-modules \[\dvr^{(K,*)} \into \dvr^K \onto \dvr\] splits by a bounded \(\dvr\)-linear section, so that \(V^K \cong V^{(K,*)} \oplus V\). Now tensor by \(A\) and use the unpointed part to conclude that \(A \hot V^{(K,*)} \cong \ker(A^K \to A) \cong A^{(K,*)}\).
\end{proof}

Now consider the simplicial subdivision functor \(\mathsf{sd} \colon \mathbb{S} \to \mathbb{S}\) and its accompanying natural transformation \(h \colon \mathsf{sd} \Rightarrow 1\) (see \cite{Goerss-Jardine:Simplicial}*{III.4} for the construction). There is an induced pro-system of simplicial sets \[\mathsf{sd}^\bullet(K) \colon \mathsf{sd}^0(K) \overset{h_K}\leftarrow \mathsf{sd}^1(K) \overset{h_{\mathsf{sd}(K)}}\leftarrow \cdots.\] The functor \(A^{(-)} \colon \mathbb{S}^\op \to \mathsf{Alg}_\dvr^\tf\) extends to one on inductive systems \(\mathsf{Ind}(\mathsf{Alg}_\dvr^\tf)\) of complete, torsionfree bornological algebras by termwise application. Applied to \(\mathsf{sd}^\bullet(K)\), we get an inductive system \(A^{\mathsf{sd}^\bullet(K)} = \setgiven{A^{\mathsf{sd}^n(K)}}{n \in \N}\) of complete, torsionfree bornological algebras. Fixing \(K\), the functor \((-)^{\mathsf{sd}^\bullet(K)} \colon \mathsf{Alg}_\dvr^\tf \to \mathsf{Ind}(\mathsf{Alg}_\dvr^\tf)\) admits an extension to the category \(\mathsf{Ind}(\mathsf{Alg}_\dvr^\tf)\). The following carries over from the algebraic setting mutatis mutandis:

\begin{lemma}\label{lem:5}
For \(A \in \mathsf{Ind}(\mathsf{Alg}_\dvr^\tf)\), the functor \(A^{\mathsf{sd}^\bullet(-)} \colon \mathbb{S}^\op \to \mathsf{Ind}(\mathsf{Alg}_\dvr^\tf)\) preserves finite limits. 
\end{lemma}

\begin{proof}
The simplicial subdivision functor \(\mathsf{sd} \colon \mathbb{S} \to \mathbb{S}\) is a left adjoint functor, so it preserves all colimits. Furthermore, for \(B \in \Alg_\dvr^\tf\), the functor \(B^{(-)}\) is a right adjoint functor, so it preserves all limits. So it takes colimits in \(\mathbb{S}^\op\) to limits in \(\Alg_\dvr^\tf\). Now if \(A = (A_i)_{i \in I} \in \mathsf{Ind}(\Alg_\dvr^\tf)\), then \[A^{\mathsf{sd}^\bullet(\mathrm{colim}_l K_l)} = \setgiven{\lim_l A_i^{\mathsf{sd}^n(K_l)}}{(i,n) \in I \times \N}\] is a limit, since finite limits are computed termwise in \(\mathsf{Ind}(\Alg_\dvr^\tf)\).   
\end{proof}

Now let \(A\) and \(B\) be inductive systems of complete, bornologically torsionfree \(\dvr\)-algebras. We can define their mapping space

\[\Hom_{\mathsf{Ind}(\Alg_\dvr^\tf)}^\bullet(A,B) \defeq ([n] \mapsto \Hom_{\mathsf{Alg}_\dvr^\tf}(A, B\gen{\Delta^n}^\updagger))\] by extending the mapping space bifunctor defined in Equation \ref{eq:mapping-space} to inductive systems of algebras. We can also define 
\[ 
\HOM_{\mathsf{Ind}(\mathsf{Alg}^\tf)}^\bullet(A,B) \defeq ([n] \mapsto  \Hom_{\mathsf{Alg}_\dvr^\tf}(A, B^{\mathsf{sd}^\bullet(\Delta^\bullet)})),
\] and the two mapping spaces are related as follows:

\begin{proposition}\label{prop:6}
Let \(A\) and \(B\) be inductive systems of complete, torsionfree bornological algebras. Then \[\mathsf{Hom}_{\mathbb{S}}(K, \HOM_{\mathsf{Ind}(\mathsf{Alg}_\dvr^\tf)}^\bullet(A,B)) \cong \Hom_{\mathsf{Ind}(\mathsf{Alg}_\dvr^\tf)}^\bullet(A,B^{\mathsf{sd}^{\bullet K}}).\] Furthermore, when \(A\) is a constant inductive system, then \(\HOM_{\mathsf{Ind}(\mathsf{Alg}_\dvr^\tf)}(A,B)\) is a fibrant simplicial set (or a Kan complex).
\end{proposition}

\begin{proof}
The proofs in \cite{Cortinas-Thom:Bivariant_K}*{Proposition 3.2.2, Theorem 3.2.3} carry over mutatis-mutandis. 
\end{proof}

We now introduce the notion of homotopy that is relevant for us. Recall that \[A\gen{\Delta^1}^\updagger = A \hot \dvr[t]^\updagger,\] where \(\dvr[t]^\updagger\) denotes the dagger completion of the polynomial ring in one variable. There is a canonical inclusion homomorphism \(\iota \colon A \to A \gen{\Delta^1}\) splitting the evaluation homomorphisms \(\ev_t \colon A \gen{\Delta^1}^\updagger \to A\) at \(t = 0,1\). An \textit{elementary homotopy} \(F \colon A_1 \to A_2\gen{\Delta^1}^\updagger\) between two bounded \(\dvr\)-algebra homomorphisms \(f_1, f_2 \colon A_0 \rightrightarrows A_1\) is a bounded \(\dvr\)-algebra homomorphism satisfying \(\ev_t \circ F = f_t\). We say that two morphisms between complete, torsionfree bornological algebras are \textit{homotopic} if they can be connected by a composition of elementary homotopies - that is, homotopy is the equivalence relation generated by elementary homotopies. Denote by \([A,B]\) the set of homotopy classes of algebra homomorphisms \(A \to B\).

Now let \(A \in \Alg_\dvr^\tf\). We define \[A^{\mathcal{S}^1} \defeq A^{(\mathsf{sd}^\bullet(S^1), *)}, \quad A^{\mathcal{S}^{n+1}} \defeq (A^{\mathcal{S}^n})^{\mathcal{S}^1},\] using which we can define the homotopy groups of the mapping space:

\begin{theorem}\label{thm:mapping space}
There are natural isomorphisms \[[A, B^{\mathcal{S}^1}] \cong \pi_1(\HOM_{\mathsf{Ind}(\mathsf{Alg}_\dvr^\tf)}^\bullet(A, B)).\]
\end{theorem}

\begin{proof}
To see that the two group structures are isomorphic, we make obvious modifications to the argument of \cite{Cortinas-Thom:Bivariant_K}*{Lemma 3.3.1}.
\end{proof}

Similarly, the Hilton-Eckmann argument implies that we can define higher homotopy groups as \([A, B^{\mathcal{S}^n}] \cong \pi_n(\mathsf{Hom}_{\mathsf{Ind}(\mathsf{Alg}_\dvr^\tf)}^\bullet(A, B))\), which are abelian for \(n \geq 2\).

\subsection{Stability}\label{subsec:stability}

Let \(X\) and \(Y\) be torsionfree bornological \(\dvr\)-modules, and let \[\gen{\cdot , \cdot} \colon Y \otimes X \to \dvr\] be a surjective \(\dvr\)-linear map. Such a map is automatically bounded. A pair \((X,Y)\) of torsionfree bornological \(\dvr\)-modules with a choice of surjection as above is called a \textit{matricial pair}. Given a matricial pair, one can define \(\mathcal{M}(X,Y)\) as the \(\dvr\)-module \(X \otimes Y\) with the product \[ (x_1 \otimes y_1) (x_2 \otimes y_2) \defeq \gen{y_1, x_2} x_1 \otimes y_2.\] This is a bounded morphism and automatically turns \(\mathcal{M}(X,Y)\) into a semi-dagger algebra. Its completion \(\mathcal{M}(X,Y)^\updagger = \comb{\mathcal{M}(X,Y)}\) is therefore a dagger algebra by \cite{Meyer-Mukherjee:Bornological_tf}*{Theorem 5.3}.

\begin{remark}
In the case of locally convex \(\C\)-algebras, the above definition recovers several reasonable notions of stability. For instance, if we take \(X = Y = \C^n\), then \(\mathcal{M}(X,Y) \cong \mathbb{M}_n(\C)\). When \(X = Y = \bigoplus_{n \in \N} \C\), we get \(\comb{\mathcal{M}(X,Y)} \cong \mathbb{M}_\infty(\C)\). For \(X = Y = l^2(\N)\), \(\comb{\mathcal{M}(X,Y)} \cong \mathcal{L}(l^2(\N))\) - algebra of trace class operators on the Hilbert space \(l^2(\N)\). Here \(\comb{\mathcal{M}(X,Y)}\) refers to the completed, projective tensor product \(X \hot Y\), with the appropriate extension of the bilinear form \(Y \otimes X \to \C\).
\end{remark}

Recall from \cite{Cortinas-Meyer-Mukherjee:NAHA}*{Section 6} that a \textit{homomorphism} between two matricial pairs \((X,Y)\) and \((W,Z)\) is a pair \(f = (f_1,f_2)\) of bounded linear maps \(f_1 \colon X \to W\) and \(f_2 \colon Y \to Z\) such that \(\gen{f_2(y), f_1(x)} = \gen{y,x}\) for all \(x \in X\) and \(y \in Y\). An \textit{elementary homotopy} is a pair \(H = (H_1,H_2)\) of bounded linear maps \(H_1 \colon X \to W[t]\) and \(H_2 \colon Y \to Z\) or \(H_1 \colon X \to W\) and \(H_2 \colon Y \to Z[t]\) such that 

\[
\begin{tikzcd}
Y \hot X \arrow{r}{H_2 \otimes H_1} \arrow{d}{\gen{\cdot,\cdot}} & Z \otimes W[t] \arrow{d}{\gen{\cdot, \cdot } \otimes \mathrm{id}} \\
\dvr \arrow{r}{\subseteq} & \dvr[t] 
\end{tikzcd}
\] commutes. Homomorphisms and homotopies of matricial pairs induce homomorphisms and dagger homotopies of algebras. We are mainly interested in homomorphisms where \(f_1 = f_2\), and we call the corresponding algebra homomorphisms \textit{standard homomorphisms}. Any pair \((x,y) \in X \times Y\) with \(\gen{y,x} = 1\) induces such a bounded algebra homomorphism \[\iota \colon V \to \mathcal{M}(X,Y)^\updagger, 1 \mapsto x \otimes y.\] Now suppose \(A\) is a complete, torsionfree bornological \(\dvr\)-algebra, there is a canonical map \[\iota_A \colon A \to A \hot \mathcal{M}(X,Y)^\updagger, \iota_A = 1_A \otimes \iota,\] where the target algebra is bornologically torsionfree by \cite{Meyer-Mukherjee:Bornological_tf}*{Proposition 4.12}. Similarly, when \(A\) is an inductive system of complete, bornologically torsionfree algebras, we can define such a canonical map, by applying the map above termwise. 

To simplify notation in the case where \(X = Y\), we denote the corresponding matrix algebra by \(\mathcal{M}_X^\updagger\). Our main cases of interest are the following matrix stabilisations:

\begin{example}\label{ex:triv:matrix}
Let \(X = Y = \bigoplus_{n \in \Lambda}\dvr\). This is the free \(\dvr\)-module on an arbitrary set \(\Lambda\), which we equip with the fine bornology. It has as basis characteristic functions \(\setgiven{\chi_n}{n \in \Lambda}\), using which we can define the bilinear form \(\gen{\chi_n,\chi_m} = \delta_{n,m}\). The corresponding matrix algebra is \(\mathbb{M}_\Lambda\), the algebra of finitely supported \(\Lambda \times \Lambda\)-matrices, which we equip with the fine bornology. We will denote this matrix algebra by \(\mathcal{M}_\Lambda^{\mathrm{alg}}\).
\end{example}

\begin{example}\label{ex:adic-matrix}
Again, let \(\Lambda\) be an arbitrary set. Now take \(X = Y = \coma{\bigoplus_{n \in \Lambda}\dvr} \cong c_0(\Lambda, \dvr)\) - the Banach \(\dvr\)-module of null sequences \(\Lambda \to \dvr\) with the supremum norm.  This is a bornological module with the bornology where every subset is bounded. The bilinear form above extends to one on \(c_0(\Lambda, \dvr) \otimes c_0(\Lambda, \dvr) \to \dvr\).   The corresponding matrix algebra \(\comb{\mathcal{M}(c_0(\Lambda,\dvr), c_0(\Lambda,\dvr))}\) is isomorphic to \(c_0(\Lambda \times \Lambda,\dvr)\) with the convolution product. We will denote this matrix algebra by \(\mathcal{M}_\Lambda^{\mathrm{cont}}\).
\end{example}

\begin{example}
  \label{exa:length-controlled_matrices}
  Let \(l \colon \Lambda \to \N\) be a proper function, that is,
  for each \(n\in\N\) the set of \(x\in\Lambda\) with
  \(l(x) \le n\) is finite.  Define~\(\dvr^{(\Lambda)}\) as in
  Example~\ref{ex:triv:matrix} and give it the bornology that
  is cofinally generated by the \(\dvr\)\nb-submodules
  \[
    S_m \defeq \sum_{\lambda\in\Lambda}
    \dvgen^{\floor{l(\lambda)/m}} \chi_\lambda
  \]
  for \(m\in\N^*\).  The bilinear form in
  Example~\ref{ex:triv:matrix} remains bounded for this
  bornology on~\(\dvr^{(\Lambda)}\).  So
  \(\mathcal{M}(\dvr^{(\Lambda)},\dvr^{(\Lambda)})\) with the tensor
  product bornology from the above bornology is a bornological
  algebra as well.  It is torsion-free and semi-dagger.  So its
  dagger completion is the same as its completion.  We denote it
  by~\(\mathcal{M}_\Lambda^l\).  It is isomorphic to the algebra of
  infinite matrices \((c_{x,y})_{x,y\in\Lambda}\) for which there is
  \(m\in\N^*\) such that
  \(c_{x,y} \in \dvgen^{\floor{(l(x)+l(y))/m}}\) for all
  \(x,y\in\Lambda\); this is the same as asking for
  \(\lim {}\abs*{c_{x,y} \dvgen^{-\floor{(l(x)+l(y))/m}}} = 0\)
  because~\(l\) is proper.  It makes no difference to replace the
  exponent of~\(\dvgen\) by \(\floor{l(x)/m}+\floor{l(y)/m}\)
  or \(\floor{\max \{l(x),l(y)\}/m}\) because we may
  vary~\(m\).
\end{example}

\begin{example}
  \label{exa:filtered_length-controlled_matrices}
  Let~\(\Lambda\) be a set with a filtration by a directed
  set~\(I\).  That is, there are subsets
  \(\Lambda_S\subseteq \Lambda\) for \(S\in I\) with
  \(\Lambda_S \subseteq \Lambda_T\) for \(S\le T\) and
  \(\Lambda = \bigcup_{S\in I} \Lambda_S\).  Let
  \(l\colon \Lambda \to \N\) be a function whose restriction
  to~\(\Lambda_S\) is proper for each \(S\in I\).  For
  \(S\in\Lambda\), form the matrix algebra~\(\mathcal{M}_{\Lambda_S}^l\) as
  in Example~\ref{exa:length-controlled_matrices}.  These algebras
  for \(S\in I\) form an inductive system.  Let
  \(\varinjlim \mathcal{M}_{\Lambda_S}^l\) be its bornological inductive
  limit.  This bornological algebra is also associated to a
  matricial pair, namely, the pair based on
  \(\varinjlim \dvr^{(\Lambda_S)}\), where each
  \(\dvr^{(\Lambda_S)}\) carries the bornology described in
  Example~\ref{exa:length-controlled_matrices}. 
\end{example}



\begin{lemma}\label{lem:matrix-homotopy}
Let \(F_0, F_1 \colon \mathcal{M}_X^\updagger \to \mathcal{M}_Y^\updagger\) be two standard homomorphisms, and let \[\iota_{\mathcal{M}_Y^\updagger} \colon \mathcal{M}_Y^\updagger \to \mathbb{M}_2 \hot \mathcal{M}_Y^\updagger\] be the canonical inclusion. Then \(\iota_{\mathcal{M}_Y^\updagger} \circ F_0\) and \(\iota_{\mathcal{M}_Y^\updagger} \circ F_1\) are dagger homotopic.   
\end{lemma} 

\begin{proof}
We first observe that \(\mathbb{M}_2 \hot \mathcal{M}_Y^\updagger \cong \mathcal{M}_{Y \hot \dvr^2}^\updagger\). Picking an orthonormal basis \(\delta_{e_1}\) and \(\delta_{e_2}\) relative to the bilinear form in Example \ref{ex:triv:matrix}, we can define a linear homotopy \(x \mapsto (1-t) F_0(x) \otimes \delta_{e_1} + t F_1(x) \otimes \delta_{e_2}\) between \(F_0 \otimes \delta_{e_1}\) and \(F_1 \otimes \delta_{e_2}\). Similarly, we can define a linear homotopy between \(F_1 \otimes \delta_{e_1} \) and \(F_1 \otimes \delta_{e_2}\). Concatenating them, we get a homotopy \(X \to Y \hot \dvr^2\) between the homomorphisms \(F_0 \otimes \delta_{e_1}\) and \(F_1 \otimes \delta_{e_1}\). This induces the required dagger homotopy between the algebra homomorphisms \(\iota_{\mathcal{M}_Y^\updagger} \circ F_0\) and \(\iota_{\mathcal{M}_Y^\updagger} \circ F_1\). 
\end{proof}

Now let \(Z\) be a torsionfree bornological \(\dvr\)-module with a nondegenerate, symmetric bilinear form \(\gen{\cdot,\cdot}_Z \to \dvr\), and \(\mathcal{M}_Z^\updagger\) its associated matrix algebra. Let \(\Gamma^Z \subseteq \mathsf{End}_\dvr(\mathcal{M}_Z^\updagger)\) be the multiplier algebra of \(\mathcal{M}_Z^\updagger\). Explicitly, this consists of pairs \((l,r)\) of right and left module maps \(\mathcal{M}_Z^\updagger \to \mathcal{M}_Z^\updagger\) such that \(x \cdot l(y) = r(x) \cdot y\) for \(x\), \(y \in \mathcal{M}_Z^\updagger\). We equip this with the equibounded bornology induced by \(\mathsf{Hom}_\dvr(\mathcal{M}_Z^\updagger, \mathcal{M}_Z^\updagger)\). Now suppose \(f_1\) and \(f_2\) are homomorphisms of the matricial pair corresponding to \(Z\), and satisfy \(\gen{f_1(x),y}_Z = \gen{x,f_2(y)}_Z\); we then call the pair \(f = (f_1,f_2)\) an \textit{adjoint pair}. Such a pair yields a multiplier via \[f_1 \cdot (x \otimes y) \defeq f_1(x) \otimes y, \quad (x \otimes y) \cdot f_2 \defeq x \otimes f_2(y).\] When we additionally have \(f_2  f_1 = 1_Z\), we call the pair an \textit{adjointable isometry}. Such a pair induces a bounded homomorphism \[\mathrm{Ad}_f \colon \mathcal{M}_Z^\updagger \to \mathcal{M}_Z^\updagger , T \mapsto f_1 \cdot T \cdot f_2.\] Note that since we assume the bilinear form \(\gen{\cdot,\cdot}_Z\) to be non-degenerate, the map \(f_2\) in an adjoint pair is unique if it exists. The symmetry assumption on the bilinear form implies that \(f_2^* = f_1\).   

Now consider the inductive system \[\mathcal{M}_{\infty,Z}^\updagger \defeq (\mathbb{M}_n)_{n \in \N} \otimes \mathcal{M}_Z^\updagger \] of complete bornological \(\dvr\)-algebras, where we equip the algebras \(\mathbb{M}_n\) with the fine bornology. At each level \(n\), we have \(\mathbb{M}_n \otimes \mathcal{M}_Z^\updagger \cong \mathcal{M}_{Z \otimes \dvr^n}^\updagger\). Tensoring with the identity on \((\mathbb{M}_n)_{n \in \N}\), we get an induced endomorphism \(1_{(\mathbb{M}_n)_n} \otimes \mathrm{Ad}_f \colon \mathcal{M}_{\infty, Z}^\updagger \to \mathcal{M}_{\infty, Z}^\updagger\). 

\begin{lemma}\label{lem:important-matrix}
Let \(f = (f_1,f_2)\) be an adjointable isometry on a torsionfree bornological \(\dvr\)-module \(Z\). Then \(\iota_{\mathcal{M}_Z^\updagger} \circ \mathrm{Ad}_f\) is dagger homotopic to \(\iota_{\mathcal{M}_Z^\updagger}\), and \(1_{(\mathbb{M}_n)_{n \in \N}} \otimes \mathrm{Ad}_f\) is dagger homotopic to the identity on \(\mathcal{M}_{\infty,Z}^\updagger\).  
\end{lemma}

\begin{proof}
The map \(\mathrm{Ad}_f\) is the standard homomorphism corresponding to the pair \((f_1,f_2)\). So by Lemma \ref{lem:matrix-homotopy}, \(\iota_{\mathcal{M}_Z^\updagger} \circ \mathrm{Ad}_f \) is dagger homotopic to \(\iota_{\mathcal{M}_Z^\updagger}\). Consequently, \(\iota_{\mathcal{M}_{\infty,Z}^\updagger} \circ  (1_{(\mathbb{M}_n)_{n \in \N}} \otimes \mathrm{Ad}_f) \colon \mathcal{M}_{\infty, Z}^\updagger \to \mathbb{M}_2 \otimes \mathcal{M}_{\infty,Z}^\updagger\) is ind-homotopic to \(\iota_{\mathcal{M}_Z^\updagger}\). Iteratively, for each \(n\), we obtain maps \[\iota_{\mathcal{M}_{\infty, Z}^\updagger}\circ (1_{(\mathbb{M}_n)_{n \in \N}} \otimes \mathrm{Ad}_f), \iota_{\mathcal{M}_{\infty, Z}^\updagger} \colon \mathcal{M}_{\infty,Z}^\updagger \rightrightarrows \mathbb{M}_{2n} \otimes \mathcal{M}_{\infty,Z}^\updagger.\]
 So it suffices to show that if \(f,g \colon A \to (\mathbb{M}_n)_{n \in \N} \otimes B\) are homomorphisms of inductive systems of complete, torsionfree bornological algebras such that \(\iota \circ f \sim \iota \circ g\), then \(f \sim g\). This follows from the same argument as in \cite{Cortinas-Thom:Bivariant_K}*{Lemma 4.1.1}. \qedhere  
\end{proof}

To see the relevance of Lemma \ref{lem:important-matrix}, consider a torsionfree bornological \(\dvr\)-module \(Z\) with a bilinear form that satisfies \(Z \oplus Z \cong Z\) and \(Z \hot Z \cong Z\), and the isomorphisms preserve the bilinear forms. This happens whenever the underlying set \(\Lambda\) in Examples \ref{ex:triv:matrix}, \ref{ex:adic-matrix}, \ref{exa:filtered_length-controlled_matrices} is infinite for any choice of set-theoretic bijection \(\Lambda \times \Lambda \cong \Lambda\). In the complex case, this happens for any separable Hilbert space. We refer to such bornological \(\dvr\)-modules \(Z\)  as \textit{product stable} bornological \(\dvr\)-modules.  We then define the \textit{direct sum} \(\oplus \colon \mathcal{M}_Z^\updagger \hot \mathcal{M}_Z^\updagger \to \mathbb{M}_2(\mathcal{M}_Z^\updagger) \cong \mathcal{M}_Z^\updagger\) and the \textit{tensor product} \(\mathcal{M}_Z^\updagger \otimes \mathcal{M}_Z^\updagger \to \mathcal{M}_{Z \hot Z}^\updagger \cong \mathcal{M}_Z^\updagger\) operations on the algebra \(\mathcal{M}_Z^\updagger\). These definitions extend to the ind-algebra \(\mathcal{M}_{\infty, Z}^\updagger\). By Lemma \ref{lem:important-matrix}, these operations are associative and commutative  up to homotopy, and the tensor product distributes over the direct sum. Consequently, we get a homotopy semi-ring \((\mathcal{M}_{\infty,Z}^\updagger , \oplus, \otimes)\).

For two inductive systems of complete, bornologically torsionfree algebras \(A\) and \(B\), we define 
\[\{A,B\} \defeq [A, \mathcal{M}_{\infty,Z}^\updagger\hot B],\] where \(Z\) is a product stable bornological \(\dvr\)-module (with a choice of bilinear map). To shorten notation, we will often denote \(\mathcal{M}_{\infty, Z}^\updagger(B) \defeq \mathcal{M}_{\infty, Z}^\updagger \hot B\). For each such choice of matrix stabilisation, we can define a category \(\mathsf{Alg}_{\mathcal{M}_{\infty, Z}^\updagger}^\tf\) whose objects are inductive systems of complete, torsionfree bornological algebras, and whose morphisms are \(\Hom_{\mathsf{Alg}_{\mathcal{M}_{\infty, Z}^\updagger}^\tf}(A,B) \defeq \{A,B\}.\) That this really is a category is shown in the following:

\begin{lemma}\label{lem:composition-matrix}
For three algebras \(A\), \(B\), \(C \in \mathsf{Ind}(\Alg_\dvr^\tf)\), we have a well-defined associative composition rule \(\{B,C\} \times \{A,B\} \to \{A,C\}\). The identity is given by the homotopy class of the map \(\iota_A\) above.  
\end{lemma}

\begin{proof}
Consider representatives of homotopy classes of maps \(f \colon A \to \mathcal{M}_{\infty, Z}^\updagger(B)\) and \(g \colon B \to \mathcal{M}_{\infty, Z}^\updagger(C)\) in \(\{A,B\}\) and \(\{B,C\}\). Let \(m \colon \mathcal{M}_{\infty, Z}^\updagger \hot \mathcal{M}_{\infty, Z}^\updagger \to \mathcal{M}_{\infty, Z}^\updagger\) be the tensor product of matrices. Then the composition \[A \overset{f}\to \mathcal{M}_{\infty,Z}^\updagger(B) \overset{1 \otimes g}\to \mathcal{M}_{\infty,Z}^\updagger \hot \mathcal{M}_{\infty,Z}^\updagger(C) \overset{m \otimes 1}\to \mathcal{M}_{\infty,Z}^\updagger \hot C,\] represents the composition \([g] \star [f]\) in \(\{A,C\}\). 
\end{proof}

We say that two algebras are \textit{matrix homotopy equivalent} if they are isomorphic in the category \(\mathsf{Alg}_{\mathcal{M}_{\infty, Z}^\updagger}^\tf\).  An algebra \(A\) is \textit{matricially stable} if it is isomorphic to \(\mathcal{M}_{\infty, Z}^\updagger(A)\). There is a canonical map \([A,B] \to \{A,B\}\), which is an isomorphism if \(B\) is matrically stable. 

We end this section by showing that an \(\mathbb{M}_2\)-stable functor acts trivially on inner endomorphisms, a well-known result in the archimedean setting (see \cite{Cuntz-Meyer-Rosenberg}*{Section 3.1.2}). Recall that a functor \(F \colon \Alg_\dvr^\tf \to \mathcal{A}\) into an abelian category is \textit{\(\mathbb{M}_2\)-stable} if the canonical map \(A \to \mathbb{M}_2(A)\) induces an isomorphism \(F(A) \to F(\mathbb{M}_2(A))\). 

\begin{proposition}\label{prop:M2-inner-endo}
Let \(F\) be an \(\mathbb{M}_2\)-stable functor, \(B \in \Alg_\dvr^\tf\) and \(A \subseteq B\) a bornological subalgebra. Suppose there are elements \(x\), \(y \in B\) such that \[y A, x A \subseteq A, \quad a xy b = ab, \quad [\dvr, x] = [\dvr, y] = 0\] for \(a\), \(b \in A\). Then \(\mathsf{Ad}(x,y) \colon A \to A\), \(a \mapsto y a x\) is a bounded \(\dvr\)-algebra homomorphism, and \(F(\mathsf{Ad}(x,y)) = \mathrm{id}_{F(A)}\). 
\end{proposition}

\begin{proof}
Let \(\iota_1\) and \(\iota_2 \colon A \to \mathbb{M}_2(A)\) denote the two corner embeddings into the upper left and lower right corners. Then \(F(\iota_1)\) is an isomorphism by assumption. Furthermore, conjugation by the matrix \(\begin{pmatrix}
0 & 1\\
-1 & 0
\end{pmatrix}\) defines an inner automorphism \(\sigma \colon \mathbb{M}_2(A) \to \mathbb{M}_2 (A)\) such that \(\sigma \circ \iota_2 = \iota_1\). Consequently, \(F(\iota_2)\) is invertible as well. Now consider the map \(\mathsf{Ad}(x \oplus \mathrm{id}, y \oplus \mathrm{id}) \colon \mathbb{M}_2(A) \to \mathbb{M}_2(A)\); it satisfies \(\mathsf{Ad}(x \oplus \mathrm{id}, y \oplus \mathrm{id}) \circ \iota_1 = \iota_1 \circ \mathsf{Ad}(x, y)\) and \(\mathsf{Ad}(x \oplus \mathrm{id}, y \oplus \mathrm{id}) \circ \iota_2 = \iota_2\). Since \(F(\iota_2)\) is invertible, the second identity says that \(F(\mathsf{Ad}(x \oplus \mathrm{id}, y \oplus \mathrm{id})) = \mathrm{id}_{F(\mathbb{M}_2(A))}\). Since \(F(\iota_1)\) is invertible, the first equality says that \(F(\mathsf{Ad}(x,y))\) is the identity on \(A\).    
\end{proof}

\subsection{Extensions of complete bornological algebras}

Let \(K \overset{f}\into E \overset{g}\onto Q\) be an extension of inductive systems of complete, bornologically torsionfree \(\dvr\)-algebras. This can be represented by a diagram \((f_\alpha \colon K_\alpha \into E_\alpha)_{\alpha}\) and \((g_\alpha \colon E_\alpha \onto Q_\alpha)_{\alpha}\) of bounded \(\dvr\)-algebra homomorphisms, where \(f_\alpha = \ker(g_\alpha)\) and \(g_\alpha = \coker(f_\alpha)\). An extension as above is called \textit{semi-split} if \(g\) has a bounded \(\dvr\)-linear section.

We can now define several canonical extensions as in \cite{Cortinas-Thom:Bivariant_K}. 

\begin{example}[Path extension]\label{ex:path-extension}
Let \(\Omega \defeq \dvr^{(S^1, *)} \cong \setgiven{\sum_{n \in \N} f \in \dvr[t]^\updagger}{f(0) = f(1) = 0} \cong t(t-1) \dvr[t]^\updagger\). By definition, this is part of an extension of complete bornological algebras 
\[ \Omega \into \dvr[t]^\updagger \overset{(\ev_0,\ev_1)}\onto \dvr \oplus \dvr.\] By Lemma \ref{lem:4}, we can tensor with a complete, bornologically torsionfree algebra \(A\) and get an extension \(\Omega (A) \into A\gen{\Delta^1}^\updagger \overset{(\ev_0, \ev_1)}\onto A \oplus A\), which we call the \textit{path extension} of \(A\). Here \(\Omega(A) = \Omega \hot A \cong t(t-1) A\gen{\Delta^1}^\updagger\). This is split by the bounded \(\dvr\)-linear section defined by \(A \oplus A \to A\gen{\Delta^1}^\updagger\), \((a_1, a_2) \mapsto (1-t)a_1 + t a_2\).  
\end{example}

Next, we come to the \textit{universal extension}. Let \(F \colon \Alg_\dvr^\tf \to \mathsf{CBorn}_\dvr^\tf\) be the canonical forgetful functor that forgets the algebra structure. This has a left adjoint given by the \textit{tensor algebra} of a complete, bornologically torsionfree \(\dvr\)-module \(\tilde{T}(M) \defeq \bigoplus_{n \in \N} M^{\hot n}\), whose multiplication is given by concatenation of pure tensors. The tensor algebra is complete and bornologically torsionfree because  \(M\) is so; this uses \cite{Meyer-Mukherjee:Bornological_tf}*{Theorem 4.6 and Proposition 4.12} and that completions and torsionfreeness are hereditary for direct sums. By termwise application, these functors extend to inductive systems of complete bornologically torsionfree \(\dvr\)-algebras and modules. We still denote them by \(\tilde{T}\) and \(F\), and their composition \(\tens  = \tilde{T} \circ F \colon \mathsf{Ind}(\Alg_\dvr^\tf) \to \mathsf{Ind}(\Alg_\dvr^\tf)\).  The free-forgetful adjuction applied to the identity on an algebra \(A \in \mathsf{Alg}_\dvr^\tf\) yields a semi-split extension \[\jens (A) \into \tens (A) \overset{\eta_A}\onto A\] of complete, torsionfree bornological algebras, where \(\jens(A) \defeq \ker(\eta_A)\). The \(\dvr\)-linear splitting is given by the obvious inclusion \(A \to \tens A\) into the first tensor factor. This extension is universal in the following sense:

\begin{lemma}\label{lem:universal-extension}
Let \(K \into E \onto A\) be a semi-split extension of complete, bornologically torsionfree \(\dvr\)-algebras. Then there is a morphism of extensions 

\[
\begin{tikzcd}
\jens (A) \arrow{r} \arrow{d}{\gamma_A} & \tens (A) \arrow{r} \arrow{d} & A \arrow{d}{1_A} \\
K \arrow{r} & E \arrow{r} & A,
\end{tikzcd}
\]
where the map \(\gamma_A\) is called the \textit{classifying map} of the extension. Furthermore, the map \(\gamma_A\) is unique up to homotopy. 
\end{lemma}

\begin{proof}
The proof of \cite{Cortinas-Thom:Bivariant_K}*{Proposition 4.4.1} works mutatis-mutandis.
\end{proof}

In general, let \(K \into E \onto B\) be any semi-split extension of complete, torsionfree bornological algebras, and \(f \colon A \to B\) an algebra homomorphism. Then composing with a choice of section \(B \to E\), we get a bounded \(\dvr\)-linear map \(A \to E\). Using the universal property of the tensor algebra, we get a bounded \(\dvr\)-algebra homomorphism \(\tens (A) \to E\) extending \(f\). This restricts to an algebra homomorphism \(\gamma_A  \colon \jens(A) \to K\). Furthermore, as in Lemma \ref{lem:universal-extension}, different choices of sections produce algebra homomorphisms \(\jens (A) \to K\) that are homotopic to \(\gamma_A\). 

The following map clarifies the functoriality of the classifying map:

\begin{proposition}\label{prop:classifying-map-functorial}
Let \[
\begin{tikzcd}
A \arrow{r} \arrow{d}{f} & B \arrow{r} \arrow{d} & C \arrow{d}{g} \\
A' \arrow{r} & B' \arrow{r} & C'
\end{tikzcd}
\] is a morphism of semi-split extensions of complete, torsionfree bornological \(\dvr\)-algebras, then there is a homotopy commuting diagram 
\[
\begin{tikzcd}
\jens(C) \arrow{r}{\gamma_C} \arrow{d}{\jens(g)} & A \arrow{d}{f} \\
\jens(C') \arrow{r}{\gamma_{C'}} & A'.
\end{tikzcd}
\]
\end{proposition}

\begin{proof}
The proof of \cite{Cortinas-Thom:Bivariant_K}*{Proposition 4.4.2} works mutatis-mutandis.
\end{proof}

\begin{lemma}\label{lem:tensoring-univ-extension}
Let \(A\) be a complete, bornologically torsionfree \(\dvr\)-algebra. Then for any complete, torsionfree bornological algebra \(R\), we have a semi-split extension 
\[ R \hot \jens(A) \into R \hot \tens(A) \onto R \hot A\] of complete, bornologically torsionfree \(\dvr\)-algebras. 
\end{lemma}

\begin{proof}
By Lemma \ref{lem:4}, tensoring by a complete bornologically torsionfree \(\dvr\)-module preserves bornological embeddings. Applying this to the universal semi-split extension \(\jens(A) \into \tens(A) \onto A\), we get the required extension that splits by the map \(1_R \otimes \sigma_A \colon R \hot A \to R \hot \tens A\), where \(\sigma_A\) is the section of \(\tens(A) \onto A\). 
\end{proof}

\begin{corollary}\label{cor:universal-simplicial-extension}
Let \(K\) be a finite pointed simplicial set. Then there is a homotopy class of maps \(\jens(A^K) \to \jens(A)^K\) represented by the following extension 
\[\jens(A)^K \into \tens(A)^K \onto A^K.\]
These maps are natural in the sense that if \(K \to L\) is a morphism of simplicial sets, then there is a homotopy commuting diagram 
\[
\begin{tikzcd}
\jens(A^K) \arrow{r} \arrow{d} & \jens(A)^K \arrow{d} \\
\jens(A^L) \arrow{r} & \jens(A)^L.
\end{tikzcd}
\]
\end{corollary}

\begin{proof}
Consider the universal tensor algebra extension \(\jens(A) \into \tens(A) \onto A\). Tensoring by \(\dvr^K\) viewed as a complete bornological \(\dvr\)-algebra with the fine bornology, we again get an extension 
\[ \jens(A) \hot \dvr^K \into \tens(A) \hot \dvr^K \onto A \hot \dvr^K \] of complete, bornologically torsionfree \(\dvr\)-algebras.
The result now follows from Lemma \ref{lem:4}.

Now suppose \(K \to L\) is a morphism of finite simplicial sets. Then the functoriality of the assignment \(K \mapsto A^K\) gives a morphism of complete bornological \(\dvr\)-algebras \(\dvr^K \to \dvr^L\). Tensoring by \(A\), we get a map \(g \colon A^K \to A^L\) using Lemma \ref{lem:4}. So we have a morphism of extensions 
\[
\begin{tikzcd}
\jens(A^K) \arrow{r} \arrow{d}{f} & \tens(A^K) \arrow{r} \arrow{d} & A^K \arrow{d}{g} \\
\jens(A^L) \arrow{r} & \tens(A^L) \arrow{r} & A^L.
\end{tikzcd}
\]
Now use Proposition \ref{prop:classifying-map-functorial} and Lemma \ref{lem:tensoring-univ-extension}.
\end{proof}

So far, we have constructed the \(J\)-functor on the homotopy category of complete, bornologically torsionfree algebras. The following lemma shows that the assignment \(\jens \colon A \to \jens A\) is compatible with matrix stabilisations, relative to any product stable bornological \(\dvr\)-module \(Z\): 

\begin{lemma}\label{lem:J-functor-matrices}
The functor \(\jens \colon \mathsf{Alg}_\dvr^\tf \to \mathsf{Alg}_\dvr^\tf\) extends to a functor on the category \(\mathsf{Alg}_{\mathcal{M}_{\infty, Z}^\updagger}^\tf\).
\end{lemma}

\begin{proof}
Any morphism in \(\mathsf{Alg}_{\mathcal{M}_{\infty, Z}^\updagger}^\tf\)  can be represented by a bounded \(\dvr\)-algebra homomorphism \(f \colon A \to \mathcal{M}_{\infty, Z}^\updagger(B)\).  This induces a map \(\jens(A) \to \jens(\mathcal{M}_{\infty, Z}^\updagger(B))\) (see Proposition \ref{prop:classifying-map-functorial}). Now apply Lemma \ref{lem:tensoring-univ-extension} to the algebra \(R = \mathcal{M}_{\infty, Z}^\updagger\), which provides a map \(\jens(\mathcal{M}_{\infty, Z}^\updagger \hot B) \to \jens(B) \hot \mathcal{M}_{\infty, Z}^\updagger\). The composition of these two maps yields the required representative \(\jens(A) \to \jens(B) \hot \mathcal{M}_{\infty, Z}^\updagger\) of a homotopy class in \([\jens A, \mathcal{M}_{\infty, Z}^\updagger \hot \jens B]\). 
\end{proof}

We now define an extension that will be instrumental in defining the mapping space for bivariant analytic \(kk\)-theory, and proving several of its properties. Let \(A\) be a complete, bornologically torsionfree \(\dvr\)-algebra. Define \(P \defeq \dvr^{(\Delta^1, *)}\) and \(\mathcal{P} \defeq \dvr^{(\mathsf{sd}^\bullet(\Delta^1), *)}\), where \(\mathsf{sd}^\bullet\) is the simplicial subdivision functor discussed before Lemma \ref{lem:5}. Then \(P \cong \ker(\dvr[t]^\updagger \overset{\ev_0}\to \dvr)\), and we have an extension of complete, bornologically torsionfree algebras 

\[ \Omega \into P \overset{\ev_1} \onto \dvr,\] which upon tensoring with \(A\) yields the semi-split extension \[\Omega(A) \into P(A) \onto A,\] called the \textit{loop extension}. Here \(\Omega(A) = \Omega \hot A\) and \(P(A) = P \hot A\), and we have used Lemma \ref{lem:4} to justify that tensoring by \(A\) is indeed exact. The \(\dvr\)-linear splitting is the one induced by \(A \subseteq A \oplus A \to A \gen{\Delta^1}^\updagger\). Here \(P(A) \defeq P \hot A\). By the universal property of the tensor algebra extension, there is a natural map  \[\varrho_A \colon \jens (A) \to \Omega(A)\] which is unique up to homotopy. In a similar manner, there is a semi-split extension \[A^{\mathcal{S}^1} \into \mathcal{P}(A) \overset{\ev_1}\onto A,\] which yields the classifying map \(\jens(A) \to A^{\mathcal{S}^1}\), obtained from the composition \(\jens(A) \to \Omega(A) \to A^{\mathcal{S}^1}\).

Now let \(f \colon A \to B\) be a bounded \(\dvr\)-algebra homomorphism. Then taking the pullback 

\begin{equation}\label{eq:mapping-path} 
\begin{tikzcd}
\Omega(B) \arrow{r} \arrow{d} & P(B) \times_B A \arrow{d} \arrow{r} & A \arrow{d}{f} \\
\Omega(B) \arrow{r} & P(B) \arrow{r}{\ev_1} & B,
\end{tikzcd}
\end{equation} we get the \textit{mapping path extension of \(f\)}.

Let us now consider the matrix algebra \(\mathcal{M}_{\N}^{\mathrm{alg}}\) from Example \ref{ex:triv:matrix}. It is the matrix algebra corresponding to the matricial pair \(X = Y = \dvr^{(\N)}\). These are the \(\dvr\)-valued functions \(\N \times \N \to \dvr\) with finite support.  It is a dagger algebra with the fine bornology. Consider the algebra \(\Gamma^\dvr\) of functions \(f \colon \N \times \N \to \dvr\) such that the set \(\setgiven{f(n,m)}{(n,m) \in \N \times \N}\) is finite, and that there is an \(N\) such that the support of the functions \(f(n,-) \colon \N \to \dvr\) and \(f(-,n) \colon \N \to \dvr\) are bounded by \(N\). This is Karoubi's cone ring over \(\dvr\). Equipping it with the fine bornology, it becomes a complete, torsionfree bornological \(\dvr\)-algebra, which we denote simply by \(\Gamma\). Furthermore, \(\Gamma\) contains \(\mathcal{M}_\N^{\mathrm{alg}}\) as a (closed) ideal, whose quotient \(\Sigma = \Gamma /\mathcal{M}_\N^{\mathrm{alg}}\) is again a complete bornological \(\dvr\)-algebra, called the \textit{suspension algebra}. 

\begin{lemma}\label{lem:suspension-torsionfree}
The suspension algebra \(\Sigma = \Gamma/\mathcal{M}_\N^{\mathrm{alg}}\) defined above is bornologically torsionfree.
\end{lemma}
\begin{proof}
Let \(\mathcal{M}^\Z\) and \(\Gamma^\Z\) denote the algebras of finitely supported functions \(\N \times \N \to \Z\), and Karoubi's cone ring over \(\Z\). They yield an extension \[\mathcal{M}_{\infty}^{\Z} \into \Gamma^\Z \onto \Sigma^\Z\] of \(\Z\)-algebras as in \cite{Cortinas-Thom:Bivariant_K}*{Equation 31}. This is also a split exact sequence of free \(\Z\)-modules. Tensoring with \(\dvr\) produces an extension of \(\dvr\)-modules \[\mathcal{M}_\infty^{\dvr} \into \Gamma^\dvr \onto \Sigma^\dvr\] that splits by a \(\dvr\)-linear map. The fine bornology functor is exact, so applying it yields an extension \[\mathcal{M}^{\mathrm{alg}} \into \Gamma \onto \Sigma\] of complete, bornological \(\dvr\)-algebras with a bounded \(\dvr\)-linear splitting. The existence of such a splitting implies that \(\Sigma\) is torsionfree, and since it has the fine bornology, it is also bornologically torsionfree.  
\end{proof}

The extension \(\mathcal{M}^{\mathrm{alg}} \into \Gamma \onto \Sigma\) is called the \textit{algebraic cone extension} of \(\dvr\). Tensoring it by a complete, torsionfree bornological \(\dvr\)-algebra \(A\), we get an extension \[\mathcal{M}^{\mathrm{alg}}(A) \into \Gamma(A) \onto \Sigma(A),\] which we call the \textit{algebraic cone extension} of \(A\).

Next we consider the matrix algebra \(\mathcal{M}_Z^\updagger\) corresponding to \(Z = \coma{\bigoplus_{n \in \N} \dvr} \cong c_0(\N)\), denoted as before by \(\mathcal{M}^{\mathrm{cont}}\). Recall from Example \ref{ex:adic-matrix} that this is the Banach \(\dvr\)-algebra \(c_0(\N \times \N)\) of functions \(\N \times \N \to \dvr\) vanishing at infinity, with the supremum norm. Now since the quotient in the extension \(\mathcal{M}_\infty^{\dvr} \into \Gamma^\dvr \onto \Sigma^\dvr\) is torsionfree, \(\dvgen\)-adic completion is exact. Consequently, we get a semi-split extension \[\coma{\mathcal{M}_\infty^\dvr} \into \coma{\Gamma^\dvr} \onto \coma{\Sigma^\dvr}\] of \(\dvgen\)-adically complete \(\dvr\)-algebras. This is also a bornological extension if we equip the algebras with the bornology where all subsets are bounded, thereby yielding a semi-split extension \(\mathcal{M}^{\mathrm{cont}} \into \coma{\Gamma} \onto \coma{\Sigma}\) of complete, bornologically torsionfree \(\dvr\)-algebras. Finally, tensoring with a complete, bornologically torsionfree \(\dvr\)-algebra \(A\), we again have an extension of complete, torsionfree bornological \(\dvr\)-algebras  \[\mathcal{M}^{\mathrm{cont}}(A) \into \coma{\Gamma}(A) \onto \coma{\Sigma}(A),\] which we call the \textit{continuous cone extension} of \(A\).

\begin{remark}\label{rem:Calkin}
The \(\dvgen\)-adically completed suspension extension defined above can be regarded as a nonarchimedean analogue of the \textit{Calkin extension} \(\mathbb{K}(l^2(\N)) \into \mathbb{B}(l^2(\N)) \onto \mathcal{Q}(l^2(\N))\). A possible alternative to this version of the Calkin extension, which resembles the \(C^*\)-algebraic version, uses the notion of bounded operators on a \(\dvgen\)-adic Hilbert space \(\dvf(X) = \setgiven{\psi \colon X \to \dvf}{\abs{\psi(x)} \leq 1 \text{ for all but finitely many } x \in X}\). This is studied in \cite{claussnitzer2019aspects}. We however find the suspension algebra defined above more conducive to the type of stability results we seek, and have proved in \cite{Cortinas-Meyer-Mukherjee:NAHA}. It is quite plausible that the \(\dvgen\)-adic Hilbert spaces and the corresponding operator algebras are special cases of matrix algebras arising from matricial pairs as in the Archimedean case. 
\end{remark}

\begin{remark}\label{rem:analytic-cone}
The reader may have noticed that we have not defined a version of the cone extension whose kernel is the matrix algebra from Example \ref{exa:filtered_length-controlled_matrices}. This is because dagger completion is \emph{not} an exact functor, so starting with the algebraic cone extension does not work. 
\end{remark}

In light of Remark \ref{rem:analytic-cone}, from now on, we restrict our attention to stabilisations by the Banach algebra \(\mathcal{M}^{\mathrm{cont}}\). 

We now move on to the \textit{Toeplitz extension}. In the algebraic case, there is a semi-split extension \[\mathcal{M}_\infty^\Z \into \mathcal{T}^\Z \onto \Z[t,t^{-1}]\] of \(\Z\)-algebras. This induces an extension of \(\dvr\)-algebras \(\mathcal{M}_\infty^\dvr \into \mathcal{T}^\dvr \onto \dvr[t,t^{-1}]\). With the fine bornology, this becomes an extension of complete, bornologically torsionfree \(\dvr\)-algebras. Finally, if \(A\) is a complete, bornologically  torsionfree \(\dvr\)-algebra, then there is an induced semi-split extension \[\mathcal{M}^{\mathrm{alg}}(A) \into \mathcal{T}(A) \onto A[t,t^{-1}]\] of complete bornologically torsionfree \(\dvr\)-algebras. Here \(\tans (A) = \tans \otimes A\) and \(A[t,t^{-1}] = A \otimes \dvr[t,t^{-1}]\). The interesting extension is, of course, the \textit{\(\dvgen\)-adically completed Toeplitz extension}, which is obtained by taking the \(\dvgen\)-adic completion \[\mathcal{M}^{\mathrm{cont}}(A) \into \coma{\tans} \hot A \onto \coma{\dvr[t,t^{-1}]} \hot A,\] of the (algebraic) Toeplitz extension, and equipping the algebras involved with the bornology where all subsets are bounded, and then finally tensoring with \(A\).  


\subsection{Free products and quasi-homomorphisms}

We now discuss the free double construction. Let \(A\) be a complete, torsionfree bornological \(\dvr\)-algebra. Then the \textit{free product} \(Q(A) \defeq A * A\) of \(A\) with itself is a complete, bornological \(\dvr\)-algebra. Note that this exists for algebras internal to any closed symmetric monoidal category with direct sums that commute with \(\otimes\). That is, for any such category \(\mathcal{C}\), it is defined by the universal property \[\mathsf{Alg}(A * B, D) \cong \mathsf{Alg}(A,D) \times \mathsf{Alg}(B,D)\] for algebras \(A\), \(B\), \(D \in \mathsf{Alg}(\mathcal{C})\).  Back to our case, denoting by \(q(A) \defeq \ker(Q(A) \to A)\), we get a split extension of complete, bornologically torsionfree \(\dvr\)-algebras \[q(A) \into Q(A) \onto A,\] where the splitting is given by each of the two canonical inclusions \(\iota_1, \iota_2 \colon A \rightrightarrows Q(A)\). Being sections, these inclusions satisfy \[\iota_1(a) - \iota_2(a) \in q(A)\] for all \(a \in A\). That is, the pair \((\iota_1, \iota_2)\) is a \textit{quasi-homomorphism} in the following sense:

\begin{definition}\label{def:quasi-homomorphism}
Let \(A\), \(B\) and \(D\) be complete, bornologically torsionfree \(\dvr\)-algebras, and let \(B \unlhd D\) be an ideal. A pair of bounded \(\dvr\)-algebra homomorphisms \(f,g \colon A \rightrightarrows D\) is called a \textit{quasi-homomorphism} if \(f(a) -g(a) \in B\) for all \(a \in A\), and the linear map \(f - g\) is bounded. 
\end{definition}

The quasi-homomorphism \((\iota_1, \iota_2)\) above is universal in following sense:

\begin{lemma}\label{lem:quasi-homomorphism-universal}
Suppose \((f,g) \colon A \to D \unrhd B\) is a quasi-homomorphism, then there is a unique bounded \(\dvr\)-algebra homomorphism \(f * g \colon Q(A) \to D\) such that the following diagram commutes:

\[
\begin{tikzcd}
A \arrow{r}{\iota_1} \arrow[r, swap, shift right, "\iota_2"] \arrow{d}{=} & Q(A)  \arrow{d}{}  & q(A) \arrow[l, swap, tail, "\unrhd"]  \arrow{d}{} \\
A \arrow{r}{f} \arrow[r, swap, shift right, "g"]  & D  & B \arrow[l,  tail, "\unrhd"]. 
\end{tikzcd}
\]
The induced map \(q(A) \to B\) is called the \textit{classifying map} of \((f,g)\). 
\end{lemma}

\begin{proof}
The pair \((f,g)\) induces a unique bounded algebra homomorphism \(Q(A) \to D\) by the universal property of free products. To describe this map explicitly, we first observe that a monomial \(a_1 \otimes b_1 \otimes \cdots \otimes a_n \otimes b_n\) is identified with \(\iota_1(a_1)\iota_2(b_1) \cdots \iota_1(a_n)\iota_2(b_n)\), and the image of alternating sums of such monomials under the maps \(\iota_1\) and \(\iota_2\) generates \(A * A\). The required map \(f*g \colon Q(A) \to D\) is defined on each such monomial by \(f*g(\iota_1(a_1)\iota_2(b_1) \cdots \iota_1(a_n)\iota_2(b_n)) \defeq f(a_1)g(b_1)\cdots f(a_n)g(b_n)\). It is bounded because \(f\) and \(g\) are bounded. By construction, this map makes the diagram above commute. Furthermore, since \(f - g\) and the multiplication map \(D \times B \to B\) are bounded, and the map \(Q(A) \hot A \to q(A)\), \(x \otimes a \mapsto x \cdot (\iota_1(a) - \iota_2(a))\) has a bounded linear section, the restriction \(q(A) \to B\) is bounded.  
\end{proof}

\section{Definition of bivariant analytic \(K\)-theory}\label{sec:definition-kk}

We now define bivariant analytic \(K\)-theory. Let \(f \colon A \to B\) be a morphism in \(\mathsf{Ind}(\mathsf{Alg}_\dvr^\tf)\). Consider the mapping path extension \[B^{\mathcal{S}^1} \into P(B) \times_B A \to A\] from Equation \eqref{eq:mapping-path}, obtained by pulling back the extension \(B^{\mathcal{S}^1} \into P(B) \onto B\) along \(f\). As this is a semi-split extension, the universal property of the tensor algebra extension \(\jens(A) \into \tens(A) \onto A\) produces a classifying map \(\jens(A) \to B^{\mathcal{S}^1}\). Using the functoriality of \(\jens\) and iterating, we get maps \[\jens^n(A) \to B^{\mathcal{S}^n} \to \mathcal{M}_{\infty,Z}^\updagger (B^{\mathcal{S}^n})\] for each \(n\). Let \([\alpha_n] \in [\jens^n(A), \mathcal{M}_{\infty,Z}^\updagger (B^{\mathcal{S}^n})]\) be a homotopy class represented by a bounded algebra homomorphism \(\jens^n(A) \to \mathcal{M}_{\infty,Z}^\updagger(B^{\mathcal{S}^n})\). Here \(Z\) is a complete, torsionfree bornological \(\dvr\)-module with a bilinear form as in Subsection \ref{subsec:stability}.  Using the tensor algebra extension again 

\[
\begin{tikzcd}
\jens^{n+1}(A) \arrow[r, tail] \arrow[d, "\alpha_{n+1}"] & \tens(\jens^n(A)) \arrow[r, two heads] \arrow{d}{} & \jens^n(A) \arrow{d}{\alpha_n} \\
\mathcal{M}_{\infty,Z}^\updagger(B^{\mathcal{S}^{n+1}}) \arrow[r, tail] & P(\mathcal{M}_{\infty,Z}^\updagger(B^{\mathcal{S}^n})) \arrow[r, two heads] & \mathcal{M}_\infty(B^{\mathcal{S}^n}), 
\end{tikzcd}
\] we get a bounded algebra homomorphism \(\alpha_{n+1}\) that represents a homotopy class in \([\jens^{n+1}(A), \mathcal{M}_\infty(B^{\mathcal{S}^n})]\). The assignment \([\alpha_n] \mapsto [\alpha_{n+1}]\) are the structure maps of an inductive system of abelian groups. 

\begin{definition}\label{def:bivariant_K}
For a fixed complete, torsionfree bornological \(\dvr\)-module \(Z\) with a bilinear form, the \textit{bivariant analytic \(K\)-theory} groups (relative to \(Z\)) are defined as the colimit \[ \kk_Z(A,B) \defeq \mathsf{colim}_n [\jens^n A, \mathcal{M}_{\infty,Z}^\updagger (B^{\mathcal{S}^n})].\] of dagger homotopy classes of bounded \(\dvr\)-algebra homomorphisms. 
\end{definition}  

We now define a category whose morphisms are given by \(\kk_Z(A,B)\) for inductive systems of complete, torsionfree bornological \(\dvr\)-algebras. Consider the endofunctors \(\jens, (-)^{\mathsf{S}^1} \colon  \mathsf{Ind}(\mathsf{Alg}_\dvr^\tf) \rightrightarrows \mathsf{Ind}(\mathsf{Alg}_\dvr^\tf)\). Recall that the loop extension \[A^{\mathcal{S}^1} \into \mathcal{P}(A) \overset{\ev_1}\onto A\] induces the classifying map \(\varrho_A \colon \jens(A) \to \mathsf{A}^{\mathcal{S}^1}\). This defines a natural transformation between the two endofunctors considered above. More concretely, if \(f \colon A \to B\) is a morphism in \(\mathsf{Ind}(\Alg_\dvr^\tf)\), we have a homotopy commutative diagram

\[
\begin{tikzcd}
\jens(A) \arrow{r}{\varrho_A} \arrow{d}{\jens(f)} & A^{\mathcal{S}^1} \arrow{d}{f^{\mathcal{S}^1}} \\
\jens(B) \arrow{r}{\varrho_B} & B^{\mathcal{S}^1},
\end{tikzcd}
\]  where homotopy commutativity means that \begin{equation}\label{eq:well-defined1}
f^{\mathcal{S}^1} \circ \varrho_A = \varrho_B \circ \jens(f) \in [\jens(A), B^{\mathcal{S}^1}].
\end{equation}
 There is a canonical map \([-,-] \to \{-,-\}\), so the same equality holds under its image in the latter group. Now consider the classifying map \(\gamma_A \colon \jens(A^{\mathcal{S}^1}) \to \jens(A)^{\mathcal{S}^1}\) constructed in Corollary \ref{cor:universal-simplicial-extension}. This induces a map 
 \begin{equation}\label{eq:well-defined2}
 \gamma_A \circ \jens(\varrho_A) = - \varrho_{\jens(A)} \colon \jens^2(A) \to \jens(A)^{\mathcal{S}^1} \in [\jens^2(A), \jens(A)^{\mathcal{S}^1}],
 \end{equation} where the equality of the two maps follows from the uniqueness of classifying maps up to homotopy. To discuss the composition rule in \(\kk_Z\), we fix some notation. For an algebra \(A\), define \[
 \gamma_A^{1,n} \defeq (\gamma_A)^{\mathcal{S}^{n-1}} \circ \dotsc \circ \gamma_{A^{\mathcal{S}^{n-2}}}^{\mathcal{S}^1} \circ \gamma_{A^{\mathcal{S}^{n-1}}} \colon \jens(A^{\mathcal{S}^{n}}) \to \jens(A)^{\mathcal{S}^{n}}\] and 
 
\begin{equation}\label{eqref:composition-kk}
\gamma_{A}^{m,n} = \gamma_{\jens^{m-1}(A)}^{1,n} \circ \dotsc \circ \jens^{m-2} \gamma_{\jens A}^{1,n} \circ \jens^{m-1} \gamma_{A}^{1,n} \colon \jens^m(A^{\mathcal{S}^n}) \to \jens^m(A)^{\mathcal{S}^n}
\end{equation} for \(m\), \(n \geq 0\).

\begin{theorem}\label{lem:composition}
Let \(A\), \(B\) and \(C \in \mathsf{Ind}(\Alg_\dvr^\tf)\). There is an associative composition product \[\kk_Z(B,C) \times \kk_Z(A,B) \to \kk_Z(A,C)\] given by extending the composition of algebra homomorphisms.
\end{theorem}

\begin{proof}
Let \([f] \in \kk_Z(A,B)\) and \([g] \in \kk_Z(B,C)\) be represented by the bounded \(\dvr\)-algebra homomorphisms \(f \colon \jens^n(A) \to \mathcal{M}_{\infty,Z}^\updagger(B^{\mathcal{S}^n})\) and \(g \colon \jens^m(B) \to \mathcal{M}_{\infty,Z}^\updagger(C^{\mathcal{S}^m})\). Their composition \([g] \circ [f]\) is represented by
\begin{multline*}
\jens^{n+m} (A) \overset{\jens^m(f)}\to \mathcal{M}_{\infty,Z}^\updagger \hot \jens^m (B^{\mathcal{S}^n}) \to  \mathcal{M}_{\infty,Z}^\updagger \hot \jens^m(B)^{\mathcal{S}^n} \\
\to \mathcal{M}_{\infty,Z}^\updagger \hot \mathcal{M}_{\infty,Z}^\updagger \hot C^{\mathcal{S}^{n+m}} \to \mathcal{M}_{\infty,Z}^\updagger(C^{\mathcal{S}^{n+m}}).
\end{multline*} 
Here the morphism \(\mathcal{M}_{\infty,Z}^\updagger \hot \jens^m(B^{\mathcal{S}^n}) \to \mathcal{M}_{\infty,Z}^\updagger \hot \jens^m(B)^{\mathcal{S}^n}\) is induced by the map \(\jens^m(B^{\mathcal{S}^n}) \to \jens^m(B)^{\mathcal{S}^n}\) from Equation \ref{eqref:composition-kk}.  Explicitly, the composition is represented by the class \([g^{\mathcal{S}^n} \circ (-1)^{mn}\gamma_B^{m,n}] \star [\jens^m(f)]\). 
That the definition of the composition does not depend on specific choices of representatives and is associative follows Equations \ref{eq:well-defined1} and \ref{eq:well-defined2} and the naturality of the transformation \(\gamma_A\) discussed in Corollary \ref{cor:universal-simplicial-extension}. 
\end{proof}

\begin{definition}\label{def:bivariant-analytic-category}
We define a category \(\kk_Z\) whose objects are complete, bornologically torsionfree \(\dvr\)-algebras, and whose morphisms are \(\kk_Z(A,B)\) for two such algebras. 
\end{definition}

There is a canonical functor \(j \colon \mathsf{Alg}_\dvr^\tf \to \kk_Z\) which acts identically on objects and associates to each morphism \(f \colon A \to B\), its image under the canonical maps \[\Hom_{\mathsf{Alg}^\tf}(A,B) \to [A,B] \to \{A,B\} \to \kk_Z(A,B).\] A morphism \(f \colon A \to B\) in \(\mathsf{Alg}^\tf\) is called a \textit{\(\kk_Z\)-equivalence} if \(j(f)\) is invertible in the category \(\kk_Z\).

In this paper, we will mostly be interested in \(\kk_Z\) for \(Z\) as in Example \ref{ex:adic-matrix}. We denote the resulting bivariant analytic \(K\)-theory simply by \(\kk\) for the rest of this paper.

\subsection{Excision}

In this subsection, we prove that \(\kk\) satisfies excision for semi-split extensions of complete, torsionfree bornological \(\dvr\)-algebras. The proof follows the same approach as \cite{Cuntz:Weyl}*{Section 5} or \cite{Cortinas-Thom:Bivariant_K}*{Section 6.3}, so we have decided to be brief in its demonstration. We first fix some notation: let \(f \colon A \to B\) be a bounded \(\dvr\)-algebra homomorphism between two such algebras. Consider the path algebra diagram \[
\begin{tikzcd}
P(B) \times_B A \arrow{r}{p_f} \arrow{d} &  A \arrow{d}{f} \\
P(B) \cong tB\gen{t}^\updagger \arrow{r}{\ev_1} & B
\end{tikzcd}
\] from Example \ref{ex:path-extension}. To shorten notation, we denote the pullback \(P(B) \times_B A\) by \(P_f\). When we use the path algebra \(\mathcal{P}(B) = B^{\mathsf{sd}^\bullet(\Delta^1)}\), we denote the resulting pullback by \(\mathcal{P}_f\). These two path algebras are \(\kk\)-equivalent. Excision in the second variable can now be stated as follows:

\begin{theorem}\label{thm:excision-1}
Let \(D \in \mathsf{Alg}_\dvr^\tf\), and let \(A \overset{f}\into B \overset{g}\onto C\) be a semi-split extension in \(\mathsf{Alg}_\dvr^\tf\). Then there is a natural long exact sequence 
\[\kk(D, \Omega B) \overset{j(\Omega(g))^*}\to \kk(D, \Omega C) \overset{\delta}\to \kk(D, A) \overset{j(f)_*}\to \kk(D, B) \overset{j(g)_*}\to \kk(D, C)\] of \(\kk\)-groups.  
\end{theorem}

\begin{proof}
By adapting the proof of \cite{Cuntz:Weyl}*{Lemma 5.1}, we see that the path extension of \(g\) yields a diagram \(\kk(D, P_g) \to \kk(D, B) \overset{j(g)_*}\to \kk(D, C)\) that is exact in the middle. Since \(g\) is linearly split, there is a \(\kk\)-equivalence between \(A\) and \(P_g\) by \cite{Cortinas-Thom:Bivariant_K}*{Lemma 6.3.2}, so that we can identify this diagram with the diagram \[\kk(D,P_g) \cong \kk(D, A) \overset{j(f)_*}\to \kk(D, B) \overset{j(g)_*}\to \kk(D,C).\] Applying the middle exactness of the path extension to the inclusion \(\iota_g \colon \Omega C \onto P_g\), we again get an extension \[kk(D, P_{\iota_g}) \to \kk(D, \Omega C) \to \kk(D,P_g) \cong \kk(D,A)\] that continues the extension above. The map \(\delta\) in the statement of the theorem is the composition \(\kk(D, \Omega C) \to \kk(D, P_g) \cong \kk(D, A)\). Now apply the analogue of \cite{Cortinas-Thom:Bivariant_K}*{Corollary 6.3.5} to the map \(g \colon B \onto C\) to get the identification \(\kk(D, \Omega B) \cong \kk(D, P_{\iota_g})\), which completes the proof.
\end{proof}

Dually, we have:

\begin{theorem}\label{thm:excision-2}
Let \(D \in \mathsf{Alg}_\dvr^\tf\), and let \(A \overset{f}\into B \overset{g}\onto C\) be a semi-split extension in \(\mathsf{Alg}_\dvr^\tf\). Then there is a natural long exact sequence 
\[\kk(C, D) \overset{j(g)_*}\to \kk(B, D) \overset{j(f)_*}\to \kk(A, D) \overset{\delta}\to \kk(\Omega C, D) \overset{j\Omega(g)_*}\to \kk(\Omega B, D)\] of \(kk\)-groups. Here \(\delta\) is the composition of \(\kk(A, D) \cong \kk(P_g, D) \to \kk(\Omega C, D)\). 
\end{theorem}

\begin{proof}
We adapt the proof of \cite{Cortinas-Thom:Bivariant_K}*{Theorem 6.3.7} to our setting. Consider a semi-split bornological quotient map \(f \colon A \to B\). Then for an \([\alpha] \in \kk(A,D)\) such that \(j(\pi_f)([\alpha]) = 0\)  in \(\kk(P_f, D)\), we can choose an \(n\) such that \(\alpha \circ \jens^n(\pi_f)\) is null-homotopic for a representative \(\alpha \colon \jens^nA \to \mathcal{M}_\infty^{\mathrm{cont}}(D^{\mathcal{S}^n})\). As a consequence, there is a bounded \(\dvr\)-algebra homomorphism \(\varphi \colon \jens^n(P_f) \to \mathcal{P}(\mathcal{M}_\infty^{\mathrm{cont}}(D^{\mathcal{S}^n})) \cong \mathcal{M}_\infty^{\mathrm{cont}}(\mathcal{P}(D^{\mathcal{S}^n}))\) that is part of the following commuting diagram 
\[
\begin{tikzcd}
\ker(\jens^n(\pi_f)) \arrow{r}{} \arrow{d} & \jens^n(P_f) \arrow{r}{\jens^n(\pi_f)} \arrow{d}{\varphi} & \jens^n(A) \arrow{d}{\alpha} \\
\mathcal{M}_\infty^{\mathrm{cont}}(D^{\mathcal{S}^{n+1}}) \arrow{r}{} &  \mathcal{M}_\infty^{\mathrm{cont}}(\mathcal{P}(D^{\mathcal{S}^n})) \arrow{r}{} & \mathcal{M}_\infty^{\mathrm{cont}}(D^{\mathcal{S}^n}).
\end{tikzcd}
\]
Now consider the composite map \[\beta \colon \jens^{n+1}(B) \to \jens^n(\Omega B) \to \ker(\jens^n(\pi_f)) \to \mathcal{M}_\infty^{\mathrm{cont}}(D^{\mathcal{S}^{n+1}}).\] Then \(\jens(f)[\beta] = [\alpha]\) by the uniqueness of the classifying map. This shows that the diagram \[\kk(B,D) \overset{j(f)_*}\to \kk(A,D) \overset{j(p_{f})_*}\to \kk(P_f, D)\] is exact in the middle. The conclusion now follows from \cite{Cortinas-Thom:Bivariant_K}*{Corollary 6.3.3, Corollary 6.3.5}, which carries over to our setting.
\end{proof}

\subsection{Looping and delooping}

Recall the \textit{loop functor} \(\Omega\) defined on objects as \(\Omega(A) \defeq \ker(P(A) \overset{\ev_1}\onto A)\) and on morphisms \(f \colon A \to B\) using the functoriality of \(\Omega\) and the canonical map \([\Omega(A), \Omega(B)] \to \kk(\Omega(A), \Omega(B))\). In this section, we promote \(\Omega\) to a functor on \(\kk\) and show that it is an equivalence of categories:

\begin{proposition}\label{prop:loop-fully-faithful}
The functor \(\Omega \colon \kk \to \kk\) is fully faithful. That is, \[\Omega \colon \kk(A,B) \to \kk(\Omega (A), \Omega (B))\] is an isomorphism of abelian groups. 
\end{proposition}

\begin{proof}
The same proof as in \cite{Cortinas-Thom:Bivariant_K}*{Lemma 6.3.8, Lemma 6.3.9} adapts to our setting. For clarity, we highlight the map in the other direction, called the \textit{delooping} map: associate to a class \([\beta] \in \kk(A^{\mathcal{S}^1}, B^{\mathcal{S}^1})\) represented by \(\beta \colon \jens^n(A^{\mathcal{S}^1}) \to B^{\mathcal{S}^{n+1}}\), the class in \(\kk(A,B)\) represented by \(\jens^{n+1}(A) \to \jens^n(A^{\mathcal{S}^1}) \overset{\beta}\to B^{\mathcal{S}^{n+1}}\). In op.cit., this map is shown to be well-defined at the level of \(kk\). \qedhere  
\end{proof}

The following is a consequence of excision:

\begin{lemma}\label{lem:noncommutative-loop}
Let \(A \in \Alg_\dvr^\tf\). Then the natural map \(\varrho_A \colon \jens A \to \Omega(A)\) induces a \(\kk\)-equivalence. 
\end{lemma}

\begin{proof}
Since the algebras \(\tens A\) and \(P A\) are contractible, Theorem \ref{thm:excision-1} implies that the boundary maps \( \kk(D, \Omega A) \to \kk(D, \jens A)\) and \(\kk(D, \Omega A)\) are isomorphisms for all \(D \in \Alg_\dvr^\tf\). By the naturality of the exact sequences in Theorem \ref{thm:excision-1}, we get a commuting diagram 
\[
\begin{tikzcd}
\kk(D, \Omega A) \arrow{r}{\cong} \arrow{d}{\mathrm{id}} & \kk(D, \jens A) \arrow{d}{\varrho_A^*} \\
\kk(D,\Omega A) \arrow{r}{\cong} & \kk(D, \Omega A),
\end{tikzcd}
\] which implies the result when we put \(D = \jens A\). 
\end{proof}

The following description of \(\kk\)-classes will be used in the subsequent sections:

\begin{lemma}\label{lem:Lemma-6.3.11}
Let \(f \colon \jens^n (A) \to B^{\mathcal{S}^n}\) denote a representative of a class in \(\kk(A,B)\). The map induced in \(\kk(\Omega^n A, \Omega^n B)\) by applying \(\Omega^n\) is given by the following composition:

\[\Omega^n(A) \overset{j(\varrho_A^n)^{-1}}\to \jens^n (A) \overset{f}\to B^{\mathcal{S}^n} \cong \Omega^n B.\]   
\end{lemma}

We now show essential surjectivity. But before we get there, we recall \textit{infinite sum rings} used in the complex operator algebraic case by Cuntz and in the algebraic case by Corti\~nas-Thom:

\begin{definition}
Let \(A\) be a complete, torsionfree bornological \(\dvr\)-algebra. 
\begin{itemize}
\item A \textit{sum algebra} is a complete, torsionfree bornological \(\dvr\)-algebra \(A\) together with distinguished elements \(\alpha_1\), \(\alpha_2\), \(\beta_1\), \(\beta_2\) satisfying \(\alpha_1 \beta_1 = \alpha_2 \beta_2 = 1\), \(\beta_1 \alpha_1 + \beta_2 \alpha_2 = 1\), and \([\alpha_i, v] = [\beta_i, v] = 0\) for all \(v \in \dvr\), \(i = 1, 2\). We denote by \(a \oplus b \defeq \beta_1 a \alpha_1 + \beta_2 b \alpha_2\);
\item Let \(B \in \mathsf{Alg}_\dvr^\tf\) and let \(\phi, \psi \colon B \rightrightarrows A\) be a bounded algebra homomorphism into a sum algebra. Let  \(\phi \oplus \psi\) be the bounded algebra homomorphism \(B \to A\) defined by \(b \mapsto \psi(b) \oplus \phi(b)\); 
\item An \textit{infinite sum \(\dvr\)-algebra} is a sum algebra \(A\) with a bounded \(\dvr\)-algebra homomorphism \(\phi^\infty \colon A \to A\) satisfying \(\phi^\infty = \mathrm{id}_A \oplus \phi^\infty\). 
\end{itemize}
\end{definition}

\begin{lemma}\label{lem:wagorer-sumring}
For a unital algebra \(A \in \mathsf{Alg}_\dvr^\tf\), the algebras \(\Gamma(A)\)  and \(\coma{\Gamma(A)}\) are infinite sum rings.
\end{lemma}

\begin{proof}
The proof in \cite{Cortinas-Thom:Bivariant_K}*{Lemma 4.8.2} shows that \((\Gamma(\dvr), \phi_\dvr^\infty)\) is an infinite sum \(\dvr\)-algebra. With the fine bornology, this is a complete, bornologically torsionfree algebra, and the homomorphism \(\phi^\infty\) is bounded. The induced map \(\coma{\phi_\dvr^\infty} \colon \coma{\Gamma(\dvr)} \to \coma{\Gamma(\dvr)}\) is a bounded algebra homomorphism in the bornology where all subsets are bounded. It satisfies \(\coma{\phi_\dvr^\infty} = \mathrm{id}_{\coma{\Gamma(\dvr)}} \oplus \coma{\phi_\dvr^\infty}\) and \(\coma{\Gamma(\dvr)}\) is an infinite sum ring with distinguished elements given by the class of the distinguished elements in the \(\dvgen\)-adic completion, which is additive and hence preserves the relations defining such elements.  Tensoring with \(A\), we again get a bounded \(\dvr\)-algebra homomorphism \(\phi_A^\infty = \coma{\phi_\dvr^\infty} \hot \mathrm{id}_A \colon \coma{\Gamma} (A) \to \coma{\Gamma}(A)\) which satisfies \(\phi^\infty(a) = a \oplus \phi^\infty(a)\). 
\end{proof}

\begin{lemma}\label{lem:technical}
Let \((A, \phi^\infty)\) be an infinite sum ring and let \(B \unlhd A\) be an ideal such that \(\phi^\infty(B) \subseteq B\). Then \(B\) is \(\kk\)-equivalent to \(0\). 
\end{lemma}
\begin{proof}
The conditions of Proposition \ref{prop:M2-inner-endo} are satisfied, and we have \(j(\mathrm{id}_B \oplus \phi_{\vert B}^\infty) = j(\mathrm{id}_B) \oplus j(\phi_{\vert B}^\infty)\). Now since \(A\) is an infinite-sum ring, we have \(j(\mathrm{id}_B \oplus \phi_{\vert B}^\infty) = j(\phi_{\vert B}^\infty)\), which shows that \(j(\mathrm{id}_B) = 0\) in \(\kk\) as required. 
\end{proof}

\begin{proposition}\label{prop:delooping}
Let \(A \in \mathsf{Alg}_\dvr^\tf\). Then \(\coma{\Sigma}\) is a delooping of \(A\). That is, we have equivalences \(\Omega \coma{\Sigma}(A) \cong A\)  in \(\kk\).
\end{proposition}

\begin{proof}
Lemma \ref{lem:wagorer-sumring} shows that for the unitalisation \(A^+ = A \oplus \dvr\), we have \(\Gamma(A^+)\) is an infinite sum ring. Next, Lemma   \ref{lem:technical} shows that for nonunital \(A\), we have that \(\coma{\Gamma}(A)\) is \(\kk\)-contractible. Now consider the extension \(\mathcal{M}^{\mathrm{cont}}(A) \into \coma{\Gamma}(A) \onto \coma{\Sigma}(A)\). By Theorem \ref{thm:excision-2}, we get that the map \[\delta_D \colon \kk(A,D) \cong \kk(\mathcal{M}_\infty^{\mathrm{cont}}(A), D) \to \kk(\Omega \coma{\Sigma}(A), D)\] is an isomorphism for each \(D \in \Alg_\dvr^\tf\). Setting \(D = A\) yields the desired result.  
\end{proof}


We can now define \(\Z\)-graded \(\kk\)-groups as follows:

\[\kk_n(A,B) = \begin{cases}
\kk(A, \Omega^n(B)) \text{ for } n \geq 0 \\
\kk(\Omega^{-n}(A), B)= \kk(A, \coma{\Sigma}^n(B)) \text{ for }  n \leq 0.
\end{cases}
\]

These are called the \textit{higher analytic \(\kk\)-groups}. They express \(\kk\) as a \(\Z\)-graded category.

\section{The universal property of \(\kk\)}\label{sec:triangulated}

In this section, we formulate the universal property of \(\kk\). Let \((\mathcal{T}, \Omega_\mathcal{T})\) be a triangulated category. 

\begin{definition}\label{def:excisive-functor} 
We say that a functor \(X \colon \Alg_\dvr^\tf \to \mathcal{T}\)  is \textit{excisive} if for any semi-split extension \(A \overset{p}\into B \overset{q}\onto C\) in \(\Alg_\dvr^\tf\), there is a map \(\delta \colon \Omega_{\mathcal{T}}X(C) \to X(A)\) satisfying the following:

\begin{itemize}
\item \(\Omega_\mathcal{T} X(C) \overset{\delta}\to X(A) \overset{X(p)}\to X(B) \overset{X(q)}\to X(C)\) is a distinguished triangle in \(\mathcal{T}\);
\item For a morphism of extensions 
\[
\begin{tikzcd}
A \arrow{r} \arrow{d}{f} & B \arrow{r} \arrow{d} & C  \arrow{d}{g} \\
A' \arrow{r} & B' \arrow{r} & C'
\end{tikzcd}
\]
the following diagram 
\[\begin{tikzcd}
\Omega X(C) \arrow{r}{\delta} \arrow{d}{\Omega X(g)} & X(A) \arrow{d}{X(f)} \\
\Omega X(C') \arrow{r}{\delta'} & X(A')
\end{tikzcd}
\] commutes. 
\end{itemize}
\end{definition}

We call a functor \(X \colon \Alg_\dvr^\tf \to \mathcal{T}\) \textit{dagger homotopy invariant} if it maps the canonical bounded algebra homomorphism \(A \to A \hot \dvr[t]^\updagger\) to an isomorphism.  It is called \textit{matricially stable} (relative to a choice of torsionfree \(\dvr\)-module \(Z\) as in subsection \ref{subsec:stability}) if it maps the canonical map \(A \to \mathcal{M}_Z^\updagger(A)\) into an isomorphism. Recall that we are primarily interested in the stabilisation relative to \(Z = \coma{\bigoplus_{n \in \N} \dvr}\), which yields the matrix algebra \(\mathcal{M}^{\mathrm{cont}}\).

By Proposition \ref{prop:loop-fully-faithful} and \ref{prop:delooping}, the loop functor \(\Omega \colon \kk \to \kk\) is an autoequivalence. As in the algebraic and complex topological case, we similarly have the following:

\begin{theorem}\label{thm:kk-triangulated}
The category \(\kk\) is a triangulated category whose distinguished triangles are diagrams isomorphic to those of the form \[\Omega(B) \to P_f \to A \overset{f}\to B,\] with auto-equivalence given by the loop functor \(\Omega \colon \kk \to \kk\).
\end{theorem}
\begin{proof}
The proof of \cite{Cortinas-Thom:Bivariant_K}*{Section 6.5} carries over mutatis-mutandis.
\end{proof}


\begin{example}\label{ex:kk-everything}
By construction, the functor \(j \colon \Alg_\dvr^\tf \to \kk\) is a dagger homotopy invariant, \(\mathcal{M}^{\mathrm{cont}}\)-stable functor. For excision, let \(A \into B \onto C\) be a semi-split extension and let \(\gamma \colon \jens C \to A\) be the classifying map. Then using the \(\kk\)-inverse of \(\varrho_C \colon \jens C \to \Omega C\) from Lemma \ref{lem:noncommutative-loop}, we get a collection \(\gamma \circ \varrho_C^{-1} \in \kk(\Omega C, A)\) as per the requirements of Definition \ref{def:excisive-functor}.
\end{example}

Adapting the proof of \cite{Cortinas-Thom:Bivariant_K}*{Theorem 6.6.2}, we have the following:

\begin{theorem}\label{thm:kk-initial}
Let \(X \colon \Alg_\dvr^\tf \to \mathcal{T}\) be a dagger homotopy invariant, \(\mathcal{M}^\mathrm{cont}\)-stable and excisive functor into a triangulated category. Then there is a unique triangulated functor \(F \colon \kk \to \mathcal{T}\) such that the following diagram 
\[
\begin{tikzcd}
\Alg_\dvr^\tf \arrow{r}{X} \arrow{d}{j} & \mathcal{T} \\
\kk \arrow[ur, swap, "F"] &
\end{tikzcd}
\] of functors commutes. 
\end{theorem}

The following are some important applications:

\begin{example}[Chern characters into periodic cyclic homology]
We first start with the Cuntz-Quillen pro-supercomplex 
\[\mathbb{HP} \colon \mathsf{Alg}(\mathsf{CBor}_\dvf) \to \mathsf{Der}(\overleftarrow{\mathsf{Ind}(\mathsf{Ban}_\dvf))})\] from the category of complete bornological \(\dvf\)-algebras into the derived category of the quasi-abelian category of pro-systems of inductive systems of Banach \(\dvf\)-vector spaces. The latter category is a triangulated category that arises as the homotopy category of a model category; this model category structure is studied in (\cite{mukherjee2022quillen}). This functor is dagger homotopy invariant (by \cite{Cortinas-Meyer-Mukherjee:NAHA}*{Theorem 4.6.1}), \(\mathcal{M}^{\cont}\)-stable and excisive (by \cite{Meyer:HLHA}*{Theorem 4.34, Section 4.3}). Since tensoring a complete torsionfree bornological \(\dvr\)-algebra with \(\dvf\) is an exact functor, the functor \[\Alg_\dvr^\tf \ni A \mapsto \mathbb{HP}(A \otimes \dvf) \in \mathsf{Der}(\overleftarrow{\mathsf{Ind}(\mathsf{Ban}_\dvf))})\] still satisfies these properties. In fact, all these properies hold for \emph{bivariant} periodic cyclic homology \(\HP_n(A,B) \defeq \Hom_{\mathsf{Der}(\overleftarrow{\mathsf{Ind}(\mathsf{Ban}_\dvf))})}(\mathbb{HP}(A), \mathbb{HP}(B)[n])\), so that by Theorem \ref{thm:kk-initial}, we obtain a triangulated functor \(\kk \to \mathsf{Der}(\overleftarrow{\mathsf{Ind}(\mathsf{Ban}_\dvf))})\) and group homomorphisms 
\[\mathrm{ch}_n \colon \kk_n(A,B) \to \HP_n(A \otimes \dvf,B \otimes \dvf)\] for all \(n \in \Z\). Setting \(A = \dvr\), we get \(\mathrm{ch}_n \colon \kk_n(\dvr, B) \to \HP_n(B \otimes \dvf)\). By \cite{Cortinas-Cuntz-Meyer-Tamme:Nonarchimedean}*{Equation 14}, when \(B\) is the dagger completion of a smooth commutative \(\dvr\)-algebra lifting of a smooth commutative \(\resf\)-algebra, we get Chern characters \(\kk_*(\dvr,B) \to \bigoplus_{j \in \Z} H_{\mathrm{rig}}^{* + 2j}(B/\dvgen B, \dvf)\), where the right hand side is periodified rigid cohomology.   
\end{example}

\begin{example}[Analytic Chern characters]\label{ex:analytic-Chern}
Now consider the homology theory defined in \cite{Cortinas-Meyer-Mukherjee:NAHA} \[\mathbb{HA} \colon \mathsf{Alg}_\dvr^\tf \to \mathsf{Der}(\overleftarrow{\mathsf{Ind}(\mathsf{Ban}_\dvf)}),\] which again satisfies dagger homotopy invariance, excision and \(\mathcal{M}^{\mathrm{cont}}\)-stability. So Theorem \ref{thm:kk-initial} again produces a triangulated functor \(\kk \to \mathsf{Der}(\overleftarrow{\mathsf{Ind}(\mathsf{Ban}_\dvf)})\) and group homomorphisms \[\mathrm{ch}_n \colon \kk_n(A,B) \to \HA_n(A,B),\] for each \(n \in \Z\). We call the group homomorphisms \(\mathrm{ch}_n\) \textit{analytic Chern characters}. Since the left hand side is an \(\dvf\)-vector space, we get \(\dvf\)-linear maps \[\mathrm{ch}_n \colon \kk_n(A,B) \otimes_\Z \dvf \to \HA_n(A,B)\] for each \(n\). When \(A = \dvr\), we get group homomorphisms \(\mathrm{ch}_n \colon \kk_n(\dvr, B) \to \HA_n(B)\). Now suppose \(B\) is fine mod \(\dvgen\) as in Definition \ref{def:fine-mod-p} - this happens when a bornological algebra is \textit{nuclear} in the sense of \cite{Meyer-Mukherjee:HL}*{Definition 3.1}. We then have \[\kk_n(\dvr, B) \otimes_\Z \dvf \to \mathrm{HL}_n(B) \cong \HA_n(B) \cong \HA_n(B/\dvgen B),\] where the right hand side is the analytic cyclic homology defined in \cite{Meyer-Mukherjee:HA}. In other words, in interesting cases, the image of the analytic Chern character depends only on the reduction mod \(\dvgen\) of the original algebra. In the next section, we will compare \(\kk_n(\dvr, B)\) with a version of \textit{analytic \(K\)-theory} defined in \cite{MR4012551} for complete, bornologically torsionfree \(\dvr\)-algebras. 
\end{example}

\begin{example}[Analytic exterior product]
Let \(B\) be a fixed complete, torsionfree bornological \(\dvr\)-algebra. Then the functor \( - \hot B \colon \Alg_\dvr^\tf \to \kk\) is excisive, homotopy invariant and stable. By the universal property of \(j \colon \Alg_\dvr^\tf \to \kk\), there is a unique extension to a triangle functor, namely \(- \hot B \colon \kk \to \kk\). As a consequence, there is an associative product 

\[
\kk(A_1, B_1) \otimes \kk(A_2, B_2) \to \kk(A_1 \hot A_2, B_1 \hot B_2)
\] for \(A_i\), \(B_j\), \(i,j = 1,2\) in \(\Alg_\dvr^\tf\). 

\end{example}

\subsection{Bivariant algebraic and analytic \(K\)-theories as stable \(\infty\)-categories}

It is by now a well-known folklore that several triangulated categories arise as homotopy categories of stable \(\infty\)-category. In this subsection, we show that the bivariant algebraic \(K\)-theory of \cite{Cortinas-Thom:Bivariant_K}, and the analytic version under consideration in this article also arise this way. Our approach is along the lines of \cite{land2018}*{Section 3}, which shows that Kasparov's \(KK\)-theory arises as the homotopy category of the simplicial localisation of separable \(C^*\)-algebras at the \(KK\)-equivalences. To adapt their result, we first fix some set-theoretic conventions. Let \(\kappa\) be an inaccessible cardinal and \(\mathcal{U}(\kappa)\) its associated Grothendieck universe. By definition, such a cardinal is larger than \(\aleph_k\) for every \(k\). We say that a category \(\mathsf{C}\) is \textit{\(\kappa\)-small} if its collection of objects and morphisms between any two objects belong to \(\mathcal{U}(\kappa)\). 

Now let \(\Alg_{\kappa}\) and \(\Alg_{\dvr,\kappa}^\tf\) be algebras and complete, torsionfree bornological \(\dvr\)-algebras with generating sets of size at most \(\kappa\). These categories are \(\kappa\)-small. Furthermore, for any two \(\kappa\)-small algebras (respectively, complete, torsionfree bornological algebras) \(A\) and \(B\), the set \([A,B]\) of polynomial (respectively, dagger) homotopy classes of  algebra homomorphisms is \(\kappa\)-small. Consequently, since the category of \(\kappa\)-small sets has \(\kappa\)-small (and hence \(\aleph_0\))-colimits, \(\kk(A,B)\) is \(\kappa\)-small. Denote by \(kk_{\kappa}\) and \(\kk_{\kappa}\) the full subcategories of \(\kappa\)-small algebras and complete, torsionfree bornological algebras, with morphisms given by \(kk_{\kappa}(A,B) = kk(A,B)\) and \(\kk_{\kappa}(A,B) = \kk(A,B)\). The same construction as in \cite{Cortinas-Thom:Bivariant_K}  yields a functor \(\Alg_\kappa \to kk_\kappa\) that is universal for the same properties as \(kk\). Let \(kk_{\infty,\kappa}\) and \(\kk_{\infty,\kappa}\) denote the localisations of the \(\infty\)-categories \(N\Alg_\kappa\) and \(N\Alg_{\dvr,\kappa}^\tf\) of \(\kappa\)-small algebras (resp. complete, torsionfree bornological \(\dvr\)-algebras) at the \(kk\) (resp. \(\kk\))-equivalences between \(\kappa\)-small algebras. Here \(N\) denotes the nerve of a category. 

\begin{proposition}\label{prop:stable-infinity-kk}
The \(\infty\)-categories \(kk_{\infty,\kappa}\) and \(\kk_{\infty,\kappa}\) are \(\kappa\)-small, stable \(\infty\)-categories. Their homotopy categories are the full subcategories \(kk_{\kappa}\) and \(\kk_{\kappa}\).
\end{proposition}

\begin{proof}
We only provide the argument for the category \(\Alg_\kappa\); the proof remains the same for the category \(\Alg_{\dvr,\kappa}^\tf\). Since the category of algebras \(\Alg_\kappa\)  is \(\kappa\)-small, so is \(N\Alg_\kappa\), and therefore so is the localisation \(\mathsf{kk}_{\infty,\kappa}\). To show that \(kk_{\infty,\kappa}\) is stable, we need to show that it is pointed, has all finite limits, and the loop functor \(\Omega \colon kk_{\infty,\kappa} \to kk_{\infty,\kappa}\) is an equivalence. 

Let \(\mathfrak{W}\) be the collection of homomorphisms \(f \colon A \to B\) between \(\kappa\)-small algebras, such that \(X(f)\) is an isomorphism for all homotopy invariant, matricially stable and excisive theories \(X \colon \mathsf{Alg}_\kappa \to \mathcal{T}\) with values in a triangulated category. Then the category \(\Alg_\kappa\) with weak equivalences given by \(\mathfrak{W}\), and  fibrations by \(\Z\)-linearly split surjections is a fibration category in the sense of \cite{brown1973abstract}. By \cite{garkusha2010algebraic}*{Theorem 9.7, 9.8}, there is a universal functor \(\mathsf{Alg}_\kappa \to \mathcal{D}\) to the derived category of the fibration category \((\Alg_\kappa, \mathfrak{W})\) that is constructed to be homotopy invariant, excisive and matricially stable. In particular, \(\mathfrak{W}\) is non-empty. Finally, by the universal property of \(kk_\kappa\), there is an equivalence of triangulated categories between \(\mathcal{D}\) and \(kk_\kappa\). The proof of \cite{land2018}*{Proposition 3.3} now goes through mutatis-mutandis.
\end{proof}

\section{Analytic \(K\)-theories for bornological \(\dvr\)-algebras}\label{sec:analytic-K}

In this section, we recall various \(K\)-theoretic constructions in the nonarchimedean setting. Here there are several different choices, depending on the homotopy type and the kind of matricial stability that is desired. For the benefit of the reader, we first list the two \(K\)-theories that are most relevant from the perspective of this article:

\begin{theorem*}
For each \(n\), there are two additive, abelian group valued functors \(\tilde{K}_n^\an, \tilde{K}_n^{\an, \updagger} \colon \Alg_\dvr^\tf \to \mathsf{Ab}\), called the \textit{stabilised} and \textit{overconvergent stabilised analytic \(K\)-theory} groups. The functors \(\tilde{K}_n^\an\) satisfy homotopy invariance for \(\coma{\dvr[t]}\)-homotopies, while \(\tilde{K}_n^{\an, \dagger}\) satisfy dagger homotopy invariance. Finally, both these functors satisfy stability for the matrix algebra \(\mathcal{M}^\cont\), and excision for semi-split extensions of complete, torsionfree bornological algebras. 
\end{theorem*}

In order to prove the theorem above, we define two intermediary functors \(K_n^\an\) and \(K_n^{\an,\updagger}\) for each \(n\). These functors are homotopy invariant for \(\coma{\dvr[t]}\) and dagger homotopies, but have no interesting stability properties in general. We study these to relate them to analytic and overconvergent analogues of the \textit{Karoubi-Villamayor \(K\)-theory} groups.

Let \(A\) be a unital, complete, torsionfree bornological \(\dvr\)-algebra. Consider the \textit{continuous path extension} \(\coma{\Omega}(A) \into \coma{P}(A) \overset{\ev_1}\onto A\), where \(\coma{P}(A) = \ker(A\gen{\Delta^1} \overset{\ev_0}\to A)\). Here \(A\gen{\Delta^1} = A \hot \dvr\gen{\Delta^1}\), where \(\dvr\gen{\Delta^n} = \coma{\dvr[t_0,\dotsc,t_n]}/\gen{\sum_{i=0}^n t_i -1}\) and we equip the Tate algebra \(\coma{\dvr[t]}\) with the bornology where all subsets are bounded.  This is a semi-split extension by complete, torsionfree bornological \(\dvr\)-algebras. The evaluation map \(\ev_1 \colon \coma{P}(A) \to A\) induces a group homomorphism \(\mathsf{GL}_n(\coma{P}(A)) \to \mathsf{GL}_n(A)\) between the groups of \(n \times n\)- invertible matrices. Taking colimits along the usual corner inclusions of matrices yields a group homomorphism \(p \colon \mathsf{GL}(\coma{P}(A)) \to \mathsf{GL}(A)\), whose image we denote by \(\mathsf{GL}(A)'\). Dividing out this subgroup, we get 

\[ KV_1^\an(A) \defeq \mathsf{GL}(A)/ \mathsf{GL}(A)'\] the first \textit{analytic \(KV\)-group} of \(A\). For non-unital algebras \(A\), we use the unitalisation \(\tilde{A} = A \oplus \dvr\), and define 
\[
KV_1^\an(A) \defeq \ker(KV_1^\an(\tilde{A}) \to KV_1^\an(\dvr)).
\]

\noindent The following properties of \(\mathrm{KV}_1^\an(A)\) will be used in the remainder of the section:

\begin{proposition}\label{prop:KV-properties}
Consider \(\mathrm{KV}_1^\an\) as a functor \(\Alg_\dvr^\tf \to \mathsf{Mod}_\Z\). 

\begin{enumerate}
\item\label{KV-quotient} There is a natural surjection \(\mathrm{K}_1(R) \onto KV_1^\an(R)\); 
\item\label{KV-split} \(KV_1^\an\) is split exact;
\item\label{KV-excision} Suppose \(A \into B \onto C\) is an extension in \(\Alg_\dvr^\tf\) such that \(\mathsf{GL}(B)' \to \mathsf{GL}(C)'\) is surjective, then there is a long exact sequence 
\[\begin{tikzcd}
KV_1^\an(A) \arrow{r}{} & KV_1^\an(B) \arrow{r}{} & KV_1^\an(C) \arrow{d}{} \\
K_0(C) & K_0(B) \arrow{l}{} & K_0(A) \arrow{l}{} 
\end{tikzcd}
\]
\item\label{KV-homotopy-invariance} \(KV_1^\an\) is additive, \(\coma{\dvr[t]}\)-homotopy invariant and \(\mathcal{M}_\N^{\mathrm{alg}}\)-stable. 
\end{enumerate}

\end{proposition}

\begin{proof}
First, let \(A\) be a complete, torsionfree bornological \(\dvr\)-algebra. Then for \(i \neq j\), \(1 + t a e_{ij}\) is a path from \(1\) to \(1 + a e_{ij}\). Since elements of the form \(1 + a e_{ij}\) generate the subgroup \(E(R) \cap \mathsf{GL}(A)\), the latter subgroup is contained in \(\mathsf{GL}(A)'\). Consequently, we get a surjection \(K_1(A) = \mathsf{GL}(A)/E(A) \cap \mathsf{GL}(A) \onto \mathsf{GL}(A)/\mathsf{GL}(A)'\) and the latter is \(KV_1^\an(A)\) by definition. Now suppose \(A\) is non-unital. Then the unital case, together with the fact that \(KV_1^\an(A) = \ker(KV_1^\an(A \oplus \dvr) \to KV_1^\an(\dvr))\) implies part \ref{KV-quotient}. To see \ref{KV-split}, let \(A \into B \onto C\) be an extension of algebras. The hypothesis that \(\mathsf{GL}(B)' \to \mathsf{GL}(C)'\) is onto implies that we get an exact sequence \(KV_1^\an(A) \to KV_1^\an(B) \to KV_1^\an(C)\). Using that the functor \(\mathsf{GL}\) and that tensoring with a complete bornological \(\dvr\)-module preserve kernels, we get a commuting diagram with exact rows and columns:
\[
\begin{tikzcd}
& 1 \arrow{d}{} & 1 \arrow{d}{} & 1 \arrow{d} \\
1 \arrow{r}{} & \mathsf{GL}(\coma{\Omega}(A)) \arrow{r}{} \arrow{d}{} & \mathsf{GL}(\coma{\Omega}(B)) \arrow{r}{} \arrow{d}{} & \mathsf{GL}(\coma{\Omega}(C)) \arrow{d}{} \\
1 \arrow{r}{} & \mathsf{GL}(\coma{P}(A)) \arrow{r}{} \arrow{d}{} & \mathsf{GL}(\coma{P}(B)) \arrow{r}{} \arrow{d}{} & \mathsf{GL}(\coma{P}(C))  \arrow{d}{} \\
1 \arrow{r}{} & \mathsf{GL}(A) \arrow{r}{}  & \mathsf{GL}(B) \arrow{r}{} & \mathsf{GL}(C). 
\end{tikzcd}
\]
Now suppose the extension \(A \into B \onto C\) is split exact. Then the rows in the diagram above are split exact. Furthermore, it follows from the diagram above that \(\mathsf{GL}(A)' = \mathsf{GL}(B)' \cap \mathsf{GL}(A)\). Consequently, \(KV_1^\an(A) \to KV_1^\an(B)\) is injective, and the sequence 

\begin{equation}\label{eqref:KV-sequence}
KV_1^\an(A) \to KV_1^\an(B) \to KV_1^\an(C)
\end{equation} is exact. Part \ref{KV-excision} follows from Part \ref{KV-quotient} and Equation \ref{eqref:KV-sequence}. For the homotopy invariance claim in \ref{KV-homotopy-invariance}, we show that the split surjection \(KV_1^\an(A\gen{\Delta^1}) \to KV_1^\an(A)\) induced by the evaluation homomorphism is an injection. By split exactness, the kernel of this map is \(KV_1^\an(PA) = \mathsf{GL}(\coma{P}A) / \mathsf{GL}(\coma{P}A)'\), so it only remains to show that \(\mathsf{GL}(\coma{P}A) \subseteq \mathsf{GL}(\coma{P}A)'\). For this, take \(\alpha(s) \in \mathsf{GL}(\coma{P}A)\). Then \(\beta(s,t) \defeq \alpha(st) \in \mathsf{GL}(\coma{P}\coma{P}A)\) is the required null-homotopy. The proof of \(\mathcal{M}_\N^\mathrm{alg}\)-stability is the similar to that of \(K_1\).    
\end{proof}

The \textit{higher analytic Karoubi-Villamayor groups} can be defined as \[KV_{n+1}^\an(A) \defeq KV_1(\coma{\Omega}^n(A)), \quad n \geq 1.\]

By Proposition \ref{prop:KV-properties}, since \(\coma{\Gamma}(R)\) is an infinite sum ring, by adapting the proof of \cite{cortinas2011algebraic}*{Proposition 2.3.1} we have that \(K_0(\coma{\Gamma}(A)) = \mathrm{KV}_1^\an(\coma{\Gamma}(A)) = 0\). Consequently, the surjection \(K_1(\coma{\Sigma}(A)) \onto \mathrm{KV}_1^\an(\coma{\Sigma}(A))\) factors through \(K_0(\coma{\Sigma}(A))\), inducing a surjection \(K_0(A) \onto \mathrm{KV}_1^\an(\coma{\Sigma} A)  \). Now applying part \ref{KV-excision} of Proposition \ref{prop:KV-properties} to the loop extension yields a map 

\begin{equation}\label{eqref:KV-K0}
\mathrm{KV}_1^\an(A) \to K_0(\coma{\Omega} (A)).
\end{equation}

Substituting \(\coma{\Sigma}(A)\) for \(A\), and composing with the map \(K_0 (A) \to \mathrm{KV}_1^\an(\coma{\Sigma}(A))\), we get a morphism \(K_0(A) \to K_0(\coma{\Sigma} \coma{\Omega} (A))\). Note that \(\coma{\Sigma}\) and \(\coma{\Omega}\) commute. Iterating this construction, we get a nonconnective version of topological \(K\)-theory in our setting:

\begin{definition}\label{def:KH-theory}
The \textit{analytic \(K\)-theory} of a complete, torsionfree bornological \(\dvr\)-algebra \(A\) is defined as the family of abelian groups 
\[K_n^\an(A) \defeq \varinjlim_m K_0(\coma{\Sigma}^m \coma{\Omega}^{n+m}(A)),\] for \(n \in \Z\). 
\end{definition}

We can also express the groups \(K_n^\an(A)\) in terms of the analytic \(KV\)-groups. To see this, we replace \(A\) by \(\coma{\Sigma}(A)\) in Equation \ref{eqref:KV-K0} and compose with the map \(K_0(-) \onto KV_1(\coma{\Sigma}(-))\) to get \[KV_1^\an(\coma{\Sigma}(A)) \to K_0(\coma{\Sigma}\coma{\Omega}(A)) \to KV_1^\an(\coma{\Sigma}^2 \coma{\Omega}(A)) = KV_2^\an(\coma{\Sigma}^2(A)).\] Iterating, we get an isomorphic inductive system whose colimit is again 
\[K_n^\an(A) \cong \varinjlim_m KV_1^\an(\coma{\Sigma}^{m+1} \coma{\Omega}^{n+m}(A)) = \varinjlim_m KV_{n+m}^\an( \coma{\Sigma}^m(A))\] for \(n \in \Z\).

As in the bornological \(\C\)-algebra case, to obtain stability of nontrivial stabilisations, we need to put this in by hand. Define the \textit{stabilised analytic \(K\)-theory} functors by 
\[\tilde{K}_n^\an(A) \defeq K_n^\an(\mathcal{M}^{\mathrm{cont}}(A)),\] for \(n \in \Z\). This is a functor on the category \(\Alg_\dvr^\tf\) of complete, torsionfree bornological \(\dvr\)-algebras. 

\begin{theorem}\label{thm:KH-theory-properties}
The functors \(\tilde{K}_n^\an\) on the category of complete, bornologically torsionfree \(\dvr\)-algebras satisfy the following properties:
\begin{enumerate}
\item Additivity;
\item \(\coma{\dvr[t]}\)-homotopy invariance;
\item \(\mathcal{M}^{\mathrm{cont}}\)-stability;
\item Excision for semi-split extensions of complete, torsionfree bornological \(\dvr\)-algebras. 
\end{enumerate}

\end{theorem}

The proof of Theorem \ref{thm:KH-theory-properties} will use the same properties of a version of negative \(K\)-theory whose domain is the category we are working in. But this causes no technical difficulties as just like with \(K_0\), we forget the topology on the algebra. Specifically, for \(n \geq 0\), we define the functors  
\[K_{-n}^{\mathrm{stab}} \colon \Alg_\dvr^\tf \to \mathsf{Mod}_\Z, \quad A \mapsto  K_0(\coma{\Sigma}^n(\mathcal{M}^{\mathrm{cont}}(A))),\] which we call the \textit{stabilised negative \(K\)-theory} groups of a bornological algebra. 

\begin{lemma}\label{lem:nonpostive-K-properties}
Let \(A \into B \onto C\) be an extension of complete, torsionfree bornological \(\dvr\)-algebras. Then for \(n \leq 0\), we have a long exact sequence 
\[
\begin{tikzcd}
K_n^{\mathrm{stab}}(A) \arrow{r}{} & K_n^{\mathrm{stab}}(B) \arrow{r}{} & K_n^{\mathrm{stab}}(C) \arrow{d}{\delta} \\
K_{n-1}^{\mathrm{stab}}(C) & K_{n-1}^{\mathrm{stab}}(B) \arrow{l}{} & K_{n-1}^{\mathrm{stab}}(A) \arrow{l}{} 
\end{tikzcd}
\]
\end{lemma}

\begin{proof}
We first see that since for any algebra \(D \in \Alg_\dvr^\tf\), \(\coma{\Gamma}(\tilde{D})\) is an infinite sum ring, the excision sequence for \(K_0\) and \(K_1\) applied to the cone extension \(\mathcal{M}^{\mathrm{cont}}(D) \into \coma{\Gamma}(D) \onto \coma{\Sigma}(D)\) yields \(K_1(\coma{\Sigma}(\tilde{D}) \cong K_0(\mathcal{M}^{\mathrm{cont}}(\tilde{D}))\). Now consider the extension  \(A \into \tilde{B} \onto \tilde{C}\) in \(\mathsf{Alg}_\dvr^\tf\). Then by Lemma \ref{lem:tensor-exact}, tensoring by \(\coma{\Sigma}\) and \(\mathcal{M}^{\mathrm{cont}}\) yields an extension \[\coma{\Sigma}(\mathcal{M}^{\mathrm{cont}}(A)) \into \coma{\Sigma}(\mathcal{M}^{\mathrm{cont}}(\tilde{B})) \onto \coma{\Sigma}(\mathcal{M}^{\mathrm{cont}}(\tilde{C}))\] again.  This yields a long exact sequence \[\begin{tikzcd}
K_0(\mathcal{M}^{\mathrm{cont}}(A)) \arrow{r}{} & K_0(\mathcal{M}^{\mathrm{cont}}(B)) \oplus K_0(\dvr) \arrow{r}{} & K_0(\mathcal{M}^{\mathrm{cont}}(C)) \oplus K_0(\dvr) \arrow{d}{\delta} \\
K_{-1}^{\mathrm{stab}}(C) \oplus K_{-1}^{\mathrm{stab}}(\dvr) & K_{-1}^{\mathrm{stab}}(B) \oplus K_{-1}^{\mathrm{stab}}(\dvr) \arrow{l}{} & K_{-1}^{\mathrm{stab}}(A), \arrow{l}{} 
\end{tikzcd}
\]
where we have used that \(K_0(\mathcal{M}^{\mathrm{cont}}) \cong K_0(\dvr)\) since \(\mathcal{M}^{\mathrm{cont}}\) is \(\dvgen\)-adically complete (see \cite{kbook}*{Lemma II.2.2}), and the \(\mathcal{M}^{\mathrm{triv}}\)-stability of \(K_0\). Splitting off the \(K_*(\dvr)\) summands, we get the result for \(n=0\). The general case follows by iteration. 
\end{proof}

\begin{proof}[Proof of Theorem \ref{thm:KH-theory-properties}]
Additivity follows from the fact that \(\coma{\Sigma}\) commutes with finite products, which is a consequence of Lemma \ref{lem:tensor-exact}. By Proposition \ref{prop:KV-properties}, \(KV_1^\an\) is \(\coma{\dvr[t]}\)-homotopy invariant. Since \(\mathrm{K}_i^\an(A) = \varinjlim_n \mathrm{KV}_{i+n}^\an(\coma{\Sigma}^n(A))\), the homotopy invariance of \(\mathrm{K}_i^\an\) follows for each \(i\). Stability follows by construction since \(\mathcal{M}^{\cont} \hot \mathcal{M}^\cont \cong \mathcal{M}^\cont\).
To see the excision claim, let \(A \into B \onto C\) be a semi-split extension of complete, torsionfree bornological \(\dvr\)-algebras. Then by Lemma \ref{lem:tensor-exact}, we see that repeated tensoring by \(\coma{\Sigma}\) and \(\coma{\Omega}\) is exact. The result now follows from the excision of stabilised non-positive \(K\)-theory in Lemma \ref{lem:nonpostive-K-properties}.  
\end{proof}

We shall now express the analytic \(KV\) and \(K\)-theory as homotopy groups of appropriate spectra. For a unital algebra \(A \in \Alg_\dvr^\tf\), denote by \(\mathsf{K}(A) = \mathsf{BGL}(A)\) its connective algebraic \(K\)-theory spectrum, and \(\mathsf{KV}^\an(A) \defeq \mathsf{K}(A\gen{\Delta^\bullet})\) the \textit{analytic Karoubi-Villamayor spectrum}. The latter is defined as the spectrum corresponding to the simplicial set \(([n] \mapsto \mathsf{K}(A \gen{\Delta^n}))\), where \(A \gen{\Delta^\bullet} \defeq A \hot \dvr\gen{\Delta^\bullet}\) is the base change with the standard rigid \(n\)-simplex. This is extended as usual to non-unital algebras by taking the homotopy fibre of the map \(\mathsf{KV}^\an(\tilde{A}) \to \mathsf{KV}^\an(\dvr)\).

Recall that the \textit{nonconnective algebraic \(K\)-theory spectrum} \(\mathbf{K}(A)\) of a unital algebra \(A\). Its \(n\)-th space is defined as 
\[\mathbf{K}(A)_n \defeq \Omega \vert \mathsf{K}(\Sigma^{n+1}(A)) \vert,\] where \(\Omega\) denotes the loop space of a topological space, and \(\vert - \vert\) is the geometric realisation of a simplicial set. 
The \textit{analytic \(K\)-theory spectrum} \(\mathbf{K}^\an(A)\) of a unital complete, torsionfree bornological \(\dvr\)-algebra \(A\) is defined as the spectrum whose \(n\)-th space is \[\mathbf{K}^\an(A)_n \defeq \Omega \mathsf{K}(\coma{\Sigma}^{n+1}(A\gen{\Delta^\bullet})) = \Omega \mathsf{K}((\coma{\Sigma}^{n+1}A)\gen{\Delta^\bullet}) = \Omega \mathsf{KV}^\an(\coma{\Sigma}^{n+1}(A)).\]

As in the purely algebraic case, the homotopy groups of the above spectra and the analytic \(KV\) and \(K\)-theory groups previous defined coincide:

\begin{theorem}\label{thm:Banach-KH}
For a complete, torsionfree bornological \(\dvr\)-algebras \(A\), we have \(\pi_n(\mathsf{KV}^\an(A)) \cong KV_n^\an(A)\) for all \(n \geq 1\), and \(\pi_n(\mathbf{K}^\an(A)) \cong K_n^\an(A)\) for all \(n \in \Z\).
\end{theorem}

\begin{proof}
The proofs of \cite{cortinas2011algebraic}*{Proposition 10.2.1} and \cite{cortinas2011algebraic}*{Proposition 10.3.2} work mutatis-mutandis. The only modification is to replace polynomial homotopies by \(\coma{\dvr[t]}\)-homotopies.\qedhere 
\end{proof}


As in \cite{tamme:thesis}*{Definition 7.4}, we also define a version of analytic \(K\)-theory that is dagger homotopy invariant, as this is the right notion of homotopy invariance for the analytic cyclic theory. To this end, we again define \(KV_1^{\an,\updagger}(A) \defeq \mathsf{GL}(A)/\mathsf{GL}(A)'\) for unital algebras, and by a completely analogous version of Theorem \ref{prop:KV-properties}, we can use split exactness to extend to unital algebras. For \(n \geq 1\), we define \(KV_{n+1}^{\an, \updagger}(A) \defeq KV_1^{\an,\updagger}(\Omega^n(A))\). Here we have used the path extension \(\Omega(A) \into P(A) \onto A\) from Example \ref{ex:path-extension}. We call the functors \(KV_n^{\an,\updagger} \colon \Alg_\dvr^\tf \to \mathsf{Mod}_\Z\) the \textit{overconvergent analytic \(K\)-theory}. The same construction as analytic \(K\)-theory applied to the overconvergent analytic \(KV\)-groups yields \[K_n^{\an,\updagger}(A) \defeq \varinjlim_m KV_{n+m+1}^{\an,\updagger}(\coma{\Sigma}^{m+1}(A))\] for \(n \in \Z\). Finally, we define the spectrum \[\mathsf{KV}^{\an,\updagger}(A) \defeq \mathsf{BGL}(A \gen{\Delta^\bullet}^\updagger),\] where \(A\) is a unital, complete, torsionfree bornological \(\dvr\)-algebra, and extend to nonunital algebras by \(\mathsf{KV}^{\an, \updagger}(A) = \mathsf{fib}(\mathsf{KV}^{\an, \updagger}(\tilde{A}) \to \mathsf{KV}^{\an, \updagger}(\dvr))\). We call this the \textit{overconvergent analytic \(KV\)-spectrum.} Its homotopy groups are denoted \(KV_n^{\an, \updagger}(A)\) and are called the \textit{overconvergent analytic \(KV\)-groups} of \(A\). Likewise, we define a nonconnective spectrum with \(n\)-th space \[\mathbf{K}^{\an, \updagger}(A)_n = \Omega \mathsf{K}(\coma{\Sigma}^{n+1}(A \gen{\Delta^\bullet}^\updagger)) = \Omega \mathsf{KV}^{\an, \updagger}(\coma{\Sigma}^{n+1}(A)),\] and extend to nonunital algebras by \(\mathbf{K}^{\an,\updagger}(A) \defeq \mathsf{fib}(\mathbf{K}(\tilde{A}) \to \mathbf{K}(\dvr))\); we call this the \textit{overconvergent analytic \(K\)-theory spectrum.} 

\begin{theorem}\label{thm:spectrum-homotopy-overconvergent}
For every complete, torsionfree bornological \(\dvr\)-algebra \(A\), we have \(\pi_n(\mathsf{KV}^{\an,\updagger}(A)) \cong KV_n^{\an,\updagger}(A)\) for all \(n \geq 1\) and \(\pi_n(\mathbf{K}^{\an,\updagger}(A)) \cong K_n^{\an,\updagger}(A)\) for all \(n \in \Z\).
\end{theorem}

\begin{proof}
Similar to proof of \cite{cortinas2011algebraic}*{Proposition 10.2.1} and \cite{cortinas2011algebraic}*{Proposition 10.3.2}. The only modification is to replace polynomial homotopies by dagger homotopies. 
\end{proof}

In what follows, let \(\tilde{K}_n^{\an, \updagger}(R) \defeq K_n^{\an, \updagger}(\mathcal{M}^{\mathrm{cont}}(R))\) for a complete, torsionfree bornological \(\dvr\)-algebra \(R\). These are called the \textit{stabilised overconvergent analytic \(K\)-groups} of \(A\).

\begin{theorem}\label{cor:KH-properties}
The stabilised overconvergent analytic \(K\)-groups \(\tilde{K}_n^{\an, \updagger} \colon \mathsf{Alg}_\dvr^\tf \to \mathsf{Mod}_\Z\) satisfy:
\begin{enumerate}
\item{\label{dagger-htpy}} Dagger homotopy invariance, that is, \(K_n^{\an, \updagger}(R) \cong K_n^{\an, \updagger}(R \hot \dvr[t]^\updagger)\) for each \(n\in \Z\);
\item{\label{matricial-stability}} \(\mathcal{M}^{\mathrm{cont}}\)-matricial stability, that is,  \(K_n^{\an, \updagger}(R) \cong K_n^{\an, \updagger}(\mathcal{M}^{\mathrm{cont}}(R))\) for each \(n \in \Z\);
\item{\label{excision}} Excision for semi-split extensions of complete, torsionfree bornological algebras: that is, for an extension \(I \into E \onto Q\) of such algebras, we have a natural long exact sequence \[\dotsc \to K_{n+1}^{\an, \updagger}(I) \to K_{n+1}^{\an, \updagger}(E) \to K_{n+1}^{\an, \updagger}(Q) \to K_n^{\an, \updagger}(I) \to K_{n}^{\an, \updagger}(E) \to K_{n}^{\an, \updagger}(Q) \to \dotsc\] of overconvergent analytic \(K\)-theory groups.
\end{enumerate}
\end{theorem}

\begin{proof}
The same proof as Theorem \ref{thm:KH-theory-properties} works after making obvious modifications. \qedhere
\end{proof}

Finally, we relate the analytic and overconvergent \(K\)-theories with the \(KV\) and \(KH\)-theories of the reduction mod \(\dvgen\).

\begin{theorem}\label{thm:reduction-mod-p}
Let \(A\) be an \(\resf\)-algebra and let \(R\) be a complete, torsionfree bornological \(\dvr\)-algebra lifting that reduces mod \(\dvgen\) to \(A\). Suppose further that \(R^\updagger \subseteq \coma{R}\). Then \(KV_n^{\an,\updagger}(R^\updagger) \cong KV_n^{\an}(\coma{R}) \cong KV_n(A)\) for \(n\geq 1\) and \(K_n^{\an,\updagger}(R^\updagger) \cong KH_n(A) \cong K_n^\an(\coma{R})\) for all \(n \in \Z\). 
\end{theorem}

\begin{proof}
The overconvergent rigid \(n\)-simplex \(\dvr\gen{\Delta^\bullet}^\updagger \) embeds into the rigid \(n\)-simplex \(\dvr\gen{\Delta^\bullet}\), which upon tensoring with the inclusion \(R^\updagger \to \coma{R}\) yields an inclusion \(R^\updagger \gen{\Delta^\bullet}^\updagger \to \coma{R} \gen{\Delta^\bullet}\). Therefore, we have maps of simplicial groups \(\mathsf{GL}(R^\updagger \gen{\Delta^\bullet}^\updagger) \to \mathsf{GL}(\coma{R} \gen{\Delta^\bullet})\). Since \(\pi_n(\mathsf{GL}(A_\bullet)) \cong \varinjlim_m \pi_n(\mathsf{GL}_m(A_\bullet))\) for a simplicial abelian group \(A_\bullet\), it suffices to prove that the induced map \[\pi_n(\mathsf{GL}_m(R^\updagger \gen{\Delta^\bullet}^\updagger)) \to \pi_n(\mathsf{GL}_m(\coma{R}\gen{\Delta^\bullet}))\] is an isomorphism for \(n \geq 0\) and any \(m \geq 1\).  

For the isomorphism claim above, the same argument as in \cite{tamme:thesis}*{Proposition 7.5} carries over. To see the surjectivity of the map in question, take a class in \(\pi_n(\mathsf{GL}_m(\coma{R}\gen{\Delta^\bullet})\); this is represented by a matrix \(g \in \mathsf{GL}_m(\coma{R}\gen{\Delta^n})\) such that \(d_i(g) = 1\) for \( i = 0,\dotsc, n\), where \(d_i\) are the face maps. By \cite{tamme:thesis}*{Lemma 7.7}, there is a sequence of matrices \((g_N) \in \mathbb{M}_m(\coma{R}\gen{\Delta^n})\) that converge to \(g\), satisfying \(d_i(g_N) = 1\) for \(i = 0,\dotsc, n\). Since \(\mathsf{GL}_m(\coma{R}\gen{\Delta^n}) \subseteq \mathbb{M}_m(\coma{R}\gen{\Delta^n})\) is open, the sequence \((g_N)\) eventually lies in \(\mathsf{GL}_m(\coma{R}\gen{\Delta^n})\). That is, for a sufficiently large \(N\), \(g_N \in \mathsf{GL}_m(\coma{R}\gen{\Delta^n})\). By \cite{tamme:thesis}*{Lemma 7.8}, it actually lies in \(\mathsf{GL}_m(R^\updagger \gen{\Delta^n}^\updagger)\). By Lemma \ref{lem:3}, \(\coma{R}\gen{\Delta^\bullet}\) and hence \(\mathbb{M}_m(\coma{R}\gen{\Delta^\bullet})\) and \(\mathbb{M}_m(\coma{R}\gen{\Delta^\bullet})^{00} \defeq ([n] \mapsto \setgiven{g \in \mathbb{M}_m(\coma{R}\gen{\Delta^n}}{\norm{g}<1})\) are contractible. Here the norm of a matrix is defined as the maximum of the norm of the entries. Now choosing \(N\) sufficiently large, we have \(g_N g^{-1} - 1 \in \mathbb{M}_m(\coma{R}\gen{\Delta^n})^{00}\). The contractibility of this simplicial set implies that there is an \(h \in \mathbb{M}_m(\coma{R}\gen{\Delta^{n+1}})^{00}\) such that \(d_0(h) = 1\) and \(d_i(h) = 0\) for all \(i \geq 1\). By the Neumann series, since \(\norm{h}<1\), we have \(1 +h \in \mathsf{GL}_m(\coma{R}\gen{\Delta^{n+1}})\). Furthermore, \(\delta_0(1 + h) = g_N g^{-1}\) and \(\delta_i(1+h) = 1\). This implies \([g_Ng^{-1}] = [1]\), or that \([g_N] = [g]\), proving surjectivity.

To see injectivity, let \(g \in \mathsf{GL}_m(R^\updagger \gen{\Delta^n}^\updagger)\) such that \(d_i(g) = 1\) for \(i = 0,\dotsc, n\). Assume there exists \(h \in \mathsf{GL}_m(\coma{R}\gen{\Delta^{n+1}})\) such that \(d_0(h) = g\) and \(d_i(h) = 0\) for \(i> 0\). Since the simplicial abelian group \(\mathbb{M}_m(R^\updagger\gen{\Delta^\bullet}^\updagger)\) is contractible, there is an \(\tilde{h} \in \mathbb{M}_m(R^\updagger\gen{\Delta^{n+1}}^\updagger)\) such that \(d_0(\tilde{h}) = g\) and \(d_i(\tilde{h}) = 1\) for \( i = 1,\dotsc, n+1\). By the same argument as for surjectivity applied to \(\tilde{h} - h\), there is a sequence \((h_N) \in \mathbb{M}_m(R^\updagger \gen{\Delta^{n+1}}^\updagger)\) converging to \(\tilde{h} - h\) such that \(\delta_i(h_N) = 0\) for \(i = 0,\dotsc, n+1\) and all \(N\). Then \(h_N + \tilde{h}\) converges to \(h \in \mathsf{GL}_m(\coma{R}\gen{\Delta^{n+1}})\). Since \(\mathsf{GL}_m(\coma{R}\gen{\Delta^{n+1}}) \subseteq \mathbb{M}_m(R\gen{\Delta^{n+1}})\) is open, applying \cite{tamme:thesis}*{Lemma 7.8}, we have \(h_N + \tilde{h} \in \mathsf{GL}_m(R^\updagger \gen{\Delta^{n+1}}^\updagger)\) for sufficiently large \(N\). As \(\delta_0(h_N + \tilde{h}) = g\) and \(\delta_i(h_N + \tilde{h}) = 1\) for \(i = 1,\dotsc, n+1\), we have  \([g] = [1]\) as required. 

What we have proved so far is that \(KV_n^{\an,\updagger}(R) \cong KV_n^{\mathrm{an}}(\coma{R})\) for \(n \geq 1\). The right hand side is isomorphic to \(KV_n(A)\) for \( n\geq 1\) by \cite{calvo}*{Proposition 2.1} (see also \cite{tamme2014karoubi}*{Remark 3.2 (ii)}). The proof uses that we have an extension of simplicial abelian groups \[1 + \dvgen \mathbb{M}_m(\coma{R}\gen{\Delta^\bullet}) \into \mathsf{GL}_m(\coma{R}\gen{\Delta^\bullet}) \onto \mathsf{GL}_m(A[\Delta^\bullet]),\] and that \(\mathbb{M}_m(\coma{R}\gen{\Delta^\bullet})\) is contractible. The exactness of the inductive limit functor now yields that \(KV_n^\an(R) \cong KV_n(A)\) for all \(n \geq 1\).  

To see the claim about \(K_n^\an\), we replace \(R\) in the argument above by \(\coma{\Sigma}(R)\). Of course, we still have \((\coma{\Sigma} \hot R)^\updagger \subseteq \coma{\Sigma} \hot \coma{R}\). Therefore, \begin{multline*}
K_n^{\an,\updagger}(R) \cong \varinjlim_m KV_{n+m}^{\an,\updagger}(\coma{\Sigma}^m(R^\updagger)) \\ \cong \varinjlim_{m} KV_{n+m}^\an(\coma{\Sigma}^m(\coma{R})) \cong \varinjlim_m KV_{n+m}(\Sigma^m (A)) \cong KH_n(A),
\end{multline*} as required. \qedhere
\end{proof}

The hypotheses of Theorem \ref{thm:reduction-mod-p} are satisfied by any Banach \(\dvr\)-algebra and any affinoid dagger algebra. For noncommutative dagger algebras that are not Banach algebras, one has to check this condition by explicitly computing the dagger completion. This has already been done for monoid algebras and crossed product algebras (see \cite{Meyer-Mukherjee:Bornological_tf}*{Section 6, Proposition 7.5}).  

\begin{corollary}\label{cor:stab-mod-p}
With \(A\) and \(R\) as in Theorem \ref{thm:reduction-mod-p}, we have \(\tilde{K}_n^{\an,\updagger} (R^\updagger) \cong K_n^{\an,\updagger}(R^\updagger) \cong KH_n(A)\). In particular, for these algebras, the stabilised and unstable overconvergent analytic \(K\)-theories coincide. 
\end{corollary}

\begin{proof}
We have \[\tilde{K}_n^{\an,\updagger}(R^\updagger) = K_n^\an(\mathcal{M}^{\mathrm{cont}} \hot R^\updagger) \cong KH_n(\mathbb{M}_\infty(A)) \cong KH_n(A),\] where the second last isomorphism follows from Theorem \ref{thm:reduction-mod-p}, and the last isomorphism follows from the \(\mathbb{M}_\infty\)-stability of homotopy algebraic \(K\)-theory. \qedhere
\end{proof}




We now have functors \(\tilde{K}_n^{\an, \updagger} \colon \Alg_\dvr^\tf \to \mathsf{Mod}_\Z\) which satisfy homotopy invariance, \(\mathcal{M}^{\mathrm{cont}}\)-stability and excision for semi-split extensions of complete, torsionfree bornological \(\dvr\)-algebras. Recall that a functor \(F \colon \Alg_\dvr^\tf \to \mathcal{A}\) from the category of complete, torsionfree bornological \(\dvr\)-algebras to an abelian category is called \textit{half-exact} (or \textit{homological}) if it maps a semi-split extension \(A \into B \onto C\) to an exact sequence \(F(A) \into F(B) \onto F(C)\). 

\begin{theorem}\label{thm:KH-kk}
Let \(F \colon \Alg_\dvr^\tf \to \mathcal{A}\) be a half-exact, \(\mathcal{M}^{\mathrm{cont}}\)-stable, additive, homotopy invariant functor. Then there is a unique homological functor \(\tilde{F} \colon \kk \to \mathcal{A}\) such that \(\tilde{F} \circ j = F\). 
\end{theorem}

\begin{proof}
Consider a semi-split extension \(E \defeq A \into B \overset{f}\onto C\) in \(\Alg_\dvr^\tf\), and let \(\iota \colon \Omega(C) \to P_f\) be the canonical inclusion. Then \(F\) sends the canonical map \(A \to P_f\) to an isomorphism, and for \(\delta_E^F = F(l)^{-1} F(\iota)\), the following sequence is exact 
\begin{equation}\label{eq:G-extension}
 F(\Omega B) \to F(\Omega C) \overset{\delta_E^F}\to F(A) \to F(B) \overset{F(f)}\to F(C).
 \end{equation} The map \(\delta_l^F\) corresponding to the loop extension \(\Omega C \into PC \onto C\) is the identity map ~\(F(\Omega C) \to F(\Omega C)\). By the homotopy invariance of \(F\), if \(B\) is contractible, then \(\delta_E^G\) above is an isomorphism.  In particular, since \(\tens C\) is contractible for any \(C \in \Alg_\dvr^\tf\), we get that \(\delta_u^F\) is an isomorphism for the universal extension of \(C\). By \ref{eq:G-extension}, we see that \(F(\varrho)\) is an isomorphism, where \(\varrho \colon \jens C \to \Omega C\) is the canonical map. Composing with the image of a representative  \(c_E \colon \jens C \to A\) in \(\kk\) of the classifying map, we get that the connecting map is \(\delta_E^F = F(c_E) F(\varrho)^{-1}\) for any semi-split extension \(E\). Since the connecting maps in \(\kk\) are \(\delta_E = c_E \circ \varrho^{-1}\), we are forced to define \(\tilde{F}(\delta_E) \defeq F(c_E)\circ F(\varrho)^{-1}\).  

Now as the proof of Theorem \ref{thm:kk-initial} shows, given a class \(\alpha \in \kk(A,B)\), there is a unique way to define \(\tilde{F}(\Omega^n\alpha)\). Let \(\tau \colon \Omega \coma{\Sigma} \to \coma{\Sigma} \Omega\) be the natural isomorphism, and let \(\delta_c \in \{\Omega \coma{\Sigma}(A), A\}\) the connecting map for the cone extension \(\mathcal{M}^{\mathrm{cone}}(A) \into \coma{\Gamma}(A) \onto \coma{\Sigma}(A)\). Then the hypotheses on \(F\) imply that \(\delta_c^F\) is an isomorphism. Furthermore, the class of \(\delta_c \tau\) in \(\kk(\coma{\Sigma} A, A)\) is a natural isomorphism. Consequently, we  must have \[\tilde{F}(\alpha) = \delta_c^F F(\tau) \tilde{F}(\Omega^n \alpha)F(\tau^{-1})(\delta_c^F)^{-1},\] which yields a functorial assignment. 

It remains to check that \(\tilde{F} \colon \kk \to \mathcal{A}\) is a homological functor. The distinguished triangles in \(\kk\) are those of the form \(\Omega A \overset{\Omega f}\to \Omega B \to P_f \to A\) for a bounded algebra homomorphism \(f \colon A \to B\). So it suffices to check that \[F(\Omega A) \overset{F(f)}\to F(\Omega B) \to F(P_f) \to F(A) \to F(B)\] is exact. This has already been checked everywhere, except at \(F(A)\), which follows from comparing the sequence above with the path sequence at \(\coma{\Sigma}\Omega f\).  
\end{proof}

By Theorem \ref{cor:KH-properties}, \(\tilde{K}_0^{\an,\updagger}\) satisfies the hypotheses of Theorem \ref{thm:KH-kk}. Consequently, there is a natural map 
\begin{equation}\label{eq:main}
\kk_0(\dvr, A) \to \Hom(K_0^{\an,\updagger}(\dvr), \tilde{K}_0^{\an,\updagger}(A)).
\end{equation}

\begin{theorem}\label{thm:kk=KH}
The map in Equation \ref{eq:main} is an isomorphism for all complete, torsionfree bornological \(\dvr\)-algebras. Consequently, we have \[\kk_n(\dvr,A) \cong \tilde{K}_n^{\an,\updagger}(A)\] for each \(n \in \Z\). 
\end{theorem}

\begin{proof}
Let \(A\) be a unital complete, torsionfree bornological \(\dvr\)-algebra. Then any class in \(K_0(A)\) comes from an idempotent \(e \in \mathcal{M}^{\mathrm{triv}}(A) \subseteq \mathcal{M}^{\mathrm{cont}}(A)\). This yields a bounded \(\dvr\)-algebra homomorphism \(\dvr \to \mathcal{M}^{\mathrm{cont}}(A)\). Since \(\kk\) is in particular \(\mathbb{M}_2\)-stable, similar idempotents yield the same map in \(\kk_0(\dvr, A)\). Consequently, we obtain a well-defined natural map \[K_0(A) \to \kk_0(\dvr, A)\] for a unital algebra \(A\). To extend this natural map to non-unital algebras in \(\Alg_\dvr^\tf\), we apply \ref{thm:excision-1} to the extension \(A \into \tilde{A} \to \dvr\). Replacing \(A\) by \(\mathcal{M}^\mathrm{cont}(A)\), the map \(K_0(A) \to \kk_0(\dvr, A)\) induces a map 
\begin{multline*}
\tilde{K}_0^{\an,\updagger}(A) = K_0^{\an,\updagger}(\mathcal{M}^{\mathrm{cont}}(A)) 
= \varinjlim_n K_0(\coma{\Sigma}^n \Omega^n(\mathcal{M}^\mathrm{cont}(A)))  \\
\longrightarrow \varinjlim_n \kk_0(\dvr, \coma{\Sigma}^n \Omega^n(\mathcal{M}^\mathrm{cont}(A))) 
= \kk_0(\dvr, \mathcal{M}^\mathrm{cont}(A)) \cong \kk_0(\dvr, A),  
\end{multline*} which we call \(\alpha\).

For the map in the other direction, let \(e \in K_0^{\an,\updagger}(\dvr) \cong  K_0(\dvr) \cong K_0(\resf)\) be the canonical generator, where the isomorphism is a special case of Proposition \ref{prop:regularity} below. By the universal property of \(\kk\) in Theorem \ref{thm:KH-kk}, there is a natural map 
\[ \beta \colon \kk_0(\dvr, A) \to \Hom( \tilde{K}_0^{\an,\updagger}(\dvr), \tilde{K}_0^{\an,\updagger}(A)) \cong \tilde{K}_0^{\an,\updagger}(A),\] which maps the class of a bounded algebra homomorphism \[\theta \colon \jens^n(\dvr) \to \mathbb{M}_r \hot \mathcal{M}^{\mathrm{cont}}(A^{\mathsf{sd}^p S^n})\] to the image of \(e\) under the map
\[\delta_l^{-n} \delta_c^n \tilde{K}_0^{\an,\updagger}(\coma{\Sigma}^n \theta) \delta_c^{-n} \delta_u^n \colon \tilde{K}_0^{\an,\updagger}(\dvr) \to \tilde{K}_0^{\an,\updagger}(A).\]
The proof that the two maps \(\alpha\) and \(\beta\) are inverse to each other goes through verbatim from the purely algebraic case (see \cite{Cortinas-Thom:Bivariant_K}*{Theorem 8.2.1}). Finally, the conclusion follows from replacing \(A\) by \(\Omega^n(A)\) for each \(n\). 
\end{proof}

\begin{corollary}\label{cor:dependence-mod-p}
Let \(A\) and \(R\) be as in Theorem \ref{thm:reduction-mod-p}. Then \(\kk_n(\dvr,R^\updagger) \cong KH_n(A)\) for each \(n \in \Z\). 
\end{corollary}

\begin{proof}
Follows from Theorem \ref{thm:kk=KH} and Corollary \ref{cor:stab-mod-p}.
\end{proof}

We end this section by comparing analytic \(K\)-theory with our version of negative \(K\)-theory for \(n \leq 0\), and the analytic Karoubi-Villamayor groups \(KV_n^\an\) for \(n \geq 0\). Recall that a ring \(A\) is called \textit{\(K_n\)-regular} if the canonical map \(A \to A[x_1,\dotsc, x_m]\) induces an isomorphism \[K_n(A) \to K_n(A[x_1,\dotsc,x_n])\] in negative \(K\)-theory for all \(m \geq 1\). Vorst's Theorem says that if a ring is \(K_0\)-regular, it is already \(K_n\)-regular for \(n \leq 0\). Examples of \(K_n\)-regular rings are, of course, regular rings. In characteristic zero, an excellent Noetherian \(k\)-algebra that is \(K_{\mathrm{dim}(A) + 1}\)-regular is regular \cites{MR2373359,kerz2021towards}. In positive characteristic, recent results \cites{kerz2021towards,geisser2012conjecture} indicate similar partial converses.

\begin{proposition}\label{prop:regularity}
Let \(R\) be a Banach \(\dvr\)-algebra algebra such that \(A=R/\dvgen R\) is \(K_n\)-regular for \(n\leq 0\). Then  
\[ K_n^{\an,\updagger}(R) = 
\begin{cases}
KV_n(R) \qquad \text{ for } n \geq 1; \\
K_0(\coma{\Sigma}^n(R)) \quad \text{ for } n \leq 0.
\end{cases}
\]
\end{proposition}

\begin{proof}
By Theorem \ref{thm:reduction-mod-p}, we have \[K_n^{\an,\updagger}(R) \cong KH_n(A) \cong \begin{cases}
KV_n (A) \qquad \text{ for } n \geq 1; \\
K_0(\Sigma^n(A)) \quad \text{ for } n \leq 0.
\end{cases},\] where the second isomorphism follows from the \(K_n\)-regularity hypothesis. Now since \(\coma{\Sigma}(R) = \coma{\Sigma} \hot R\) is \(\dvgen\)-adically complete, by \cite{kbook}*{Lemma II.2.2}, we have \(K_0(\coma{\Sigma}(R)) \cong K_0(\Sigma(A))\). 
\end{proof}

\section{Chern characters by lifting to \(\dvf\)-vector spaces}

We conclude this article by essentially summarising the different relationships between bivariant \(K\)-theory and analytic cyclic homology for \(\resf\)-algebras and torsionfree \(\dvr\)-algebras.

\subsection{The Chern character on bivariant algebraic \(K\)-theory:} Recall that the algebraic bivariant \(K\)-theory constructed in \cite{Cortinas-Thom:Bivariant_K} is the universal functor \(j \colon \mathsf{Alg}_l \to kk\) into a triangulated category satisfying polynomial homotopy invariance, \(\mathbb{M}_\infty\)-stability and excision. Here \(l\) is a commutative, unital  ring. In particular, it applies to the case we are interested in, namely the residue field \(l = \resf\) of the discrete valuation ring \(\dvr\). The correct target for this is the analytic cyclic homology complex \(\HAC \colon \mathsf{Alg}_\resf \to \overleftarrow{\mathsf{Kom}(\mathsf{Ind}(\mathsf{Ban}_\dvf))}\), which satisfies all the afforementioned properties. By the universal property of \(kk\), we get group homomorphisms 
\[
\mathrm{ch}_n \colon kk_n(A,B) \to \HA_n(A,B)  
\] for \(A\), \(B \in \mathsf{Alg}_\resf\), \(n \in \Z\). When \(A = \resf\), the left hand side is Weibel's homotopy \(K\)-theory \(KH_*(B)\) by \cite{Cortinas-Thom:Bivariant_K}*{Theorem 8.2.1}, while the right hand side is the analytic cyclic homology \(\HA_*(B)\) by \cite{Meyer-Mukherjee:HA}*{Theorem 3.10}. So we get group homomorphisms \(KH_n(B) \to \HA_n(B)\) for each \(n \in \Z\). Since the right hand side consists of \(\dvf\)-vector spaces, we summarily have \(\dvf\)-linear maps 
\begin{equation}\label{eq:algebraic-chern}
KH_n(B) \otimes_\Z \dvf \to \HA_n(B)
\end{equation} for each \(n \in \Z\).

As already mentioned in Example \ref{ex:analytic-Chern}, by the universal property of \(\kk\) and the properties of the functor \(\HAC \colon \Alg_\dvr^\tf \to \overleftarrow{\mathsf{Kom}(\mathsf{Ind}(\mathsf{Ban}_\dvf))}\), we get group homomorphisms \[\mathrm{ch}_n \colon \kk_n(R,S) \to \HA_n(R,S)\] for each \(n \in \Z\) and \(R\), \(S \in \Alg_\dvr^\tf\). Setting \(R = \dvr\), Theorem \ref{thm:kk=KH} and \cite{Cortinas-Meyer-Mukherjee:NAHA}*{Section 3.1} yield group homomorphisms \[\tilde{K}_n^{\an, \updagger}(S) \to \HA_n(S)\] for each \(n \in \Z\). Since the right hand side is an \(\dvf\)-vector space, we get \(\dvf\)-linear maps 
\begin{equation}\label{eq:analytic-chern-char}
\mathrm{ch}_n \otimes \dvf \colon \tilde{K}_n^{\an, \updagger}(S) \otimes_\Z \dvf \to \HA_n(S)
\end{equation} for each \(n \in \Z\).

\begin{remark}\label{bootstrap-class}
In the complex topological case, we get similar Chern characters \[\mathrm{ch}_n^{\mathrm{top}} \colon K_n^{\mathrm{top}}(A) \otimes_\Z \C \to \mathrm{HL}_n(A)\] from topological \(K\)-theory map to local cyclic homology. This map is an isomorphism for separable \(C^*\)-algebras in the \(C^*\)-algebraic bootstrap class (see \cite{Meyer:HLHA}*{Theorem 7.7}). This is unlikely to be true in the nonarchimedean case because the left hand side could have nontrivial (and non-isomorphic) groups for each \(n \in \Z\), while the right hand side is \(2\)-periodic by construction. To address this, we take the product periodification \(\tilde{K}^{\an, \updagger}(S)_{\ev} = \prod_{n \in \Z} \tilde{K}_{2n}^{\an, \updagger}(S)\) (resp. \(\tilde{K}^{\an, \updagger}(S)_{\odd} = \prod_{n \in \Z} \tilde{K}_{2n+1}^{\an, \updagger}(S)\)) on the left hand side and get maps \[\tilde{K}^{\an, \updagger}(S)_\ev \to \HA_0(S) \quad \text{ and  } \quad  \tilde{K}^{\an, \updagger}(S)_{\odd} \to \HA_1(S).\] 
\end{remark}

\subsection{From analytic \(K\)-theory to its reduction mod \(\dvgen\)}

The reduction mod \(\dvgen\) of a torsionfree bornological \(\dvr\)-algebra \(\Alg_\dvr^\tf \overset{\otimes_\dvr \resf}\to \Alg_\resf\) induces an obvious functor \(\kk \to kk\). On the cyclic homology side, suppose \(A\) is an \(\resf\)-algebra and \(D\) is a dagger algebra that is fine mod \(\dvgen\) and satisfies \(D/\dvgen D \cong A\). When \(A\) is smooth, commutative, then there always exists a smooth, commutative \(\dvr\)-algebra lifting \(R\) such that \(R/\dvgen R \cong A\). Taking the dagger completion and equipping it with the compactoid bornology ensures that the quotient bornology on \(A\) is fine. Once we have such a dagger algebra lifting that is fine mod \(\dvgen\), we get a weak equivalence \(\HAC(D) \cong \HAC(A)\) by \cite{Meyer-Mukherjee:HA}*{Theorem 5.5}, for the model structure constructed in \cite{mukherjee2022quillen}. As a consequence, we get \(\HA_n(R, S) \cong \HA_n(A, B)\) for each \(n \in \Z\), where \(R\) and \(S\) are dagger algebra liftings that are fine mod \(\dvgen\). Summarily, we have a diagram 
\[
\begin{tikzcd}
\kk_n(R,S) \arrow{r}{\mathrm{ch}_n} \arrow{d}{\otimes_\dvr \resf} & \HA_n(R,S) \arrow{d}{\cong} \\
kk_n(A,B) \arrow{r}{\mathrm{ch}_n} & \HA_n(A,B)
\end{tikzcd} 
\] of abelian groups for each \(n\in \Z\). Setting \(R = \dvr\) and \(A = \resf\), we get 
\begin{equation}\label{eq:kk-HA-chern}
\begin{tikzcd}
\tilde{K}_n^{\an,\updagger}(S) \arrow{r}{\mathrm{ch}_n} \arrow{d} & \HA_n(S) \arrow{d}{\cong} \\
KH_n(B) \arrow{r}{\mathrm{ch}_n} & \HA_n(B)
\end{tikzcd} 
\end{equation} for each \(n \in \Z\). By Theorem \ref{thm:reduction-mod-p}, if \(S\) is contained in its \(\dvgen\)-adic completion, the vertical map of \ref{eq:kk-HA-chern} is an isomorphism.

\begin{remark}
Referring again to Remark \ref{bootstrap-class}, a natural question again arises when the periodified Chern character \[\prod_{n \in \Z} KH_{2n}(B) \otimes_\Z \dvf \to \HA_0(B) \quad \text{ and } \prod_{n \in \Z} KH_{2n+1}(B) \otimes_\Z \dvf \to \HA_1(B)\] is an isomorphism for \(\resf\)-algebras \(B\). For \(\resf = \resf_p\) - the finite field with \(p\)-elements, since algebraic \(K\)-theory of \(\resf\) has \(p\)-torsion for all higher algebraic \(K\)-theory groups, the above isomorphism holds. More generally, we expect the result to hold for all algebras in the \textit{algebraic bootstrap class}, which is defined as the triangulated subcategory of \(kk\) generated by \(\resf\). Finally, if we consider algebras \(A\) in the algebraic bootstrap class that admit dagger algebra liftings \(D\) that reduce mod \(\dvgen\) to \(A\) with the fine bornology, in light of the diagram \ref{eq:kk-HA-chern}, we should get isomorphisms
\begin{multline*}
K^{\an, \updagger}(D)_\ev \cong \prod_{n \in \Z} KH_{2n}(A) \to \HA_0(A) \cong \HA_0(D) \\
K^{\an, \updagger}(D)_\odd \cong \prod_{n \in \Z} KH_{2n+1}(A) \to \HA_1(A) \cong \HA_1(D)
\end{multline*} of \(\dvf\)-vector spaces. The algebraic bootstrap class for algebras over arbitrary commutative rings is presently being investigated by Guillermo Corti\~nas.
\end{remark}




\end{document}